\documentclass[10pt]{amsart}
\usepackage[utf8]{inputenc} 
\usepackage[T1]{fontenc}
\usepackage{amsfonts,amscd}
\usepackage[a4paper, margin=1.25in]{geometry}
\usepackage[parfill]{parskip}
\usepackage{xcolor}
\usepackage{comment}
\usepackage{amssymb,latexsym,amsmath,amsthm,enumitem,mathrsfs}
\usepackage{fullpage}
\usepackage{fancyhdr}
\usepackage{hyperref}
\usepackage{diagbox}
\usepackage{array}
\usepackage{tikz}
\usetikzlibrary{arrows}
\usetikzlibrary{positioning}
\usepackage{float}
\usepackage{mathtools} 
\usepackage{dsfont}
\usepackage[titletoc,title]{appendix}
\newtheorem{theorem}{Theorem}
\newtheorem{defn}[theorem]{Definition}

\newtheorem{lemma}[theorem]{Lemma}

\DeclareRobustCommand{\rchi}{{\mathpalette\irchi\relax}}
\newcommand{\irchi}[2]{\raisebox{\depth}{$#1\chi$}} 

\mathtoolsset{showonlyrefs}

\allowdisplaybreaks
\usepackage{graphicx} 

\title[Symmetric power $L$-function coefficients over sums of squares]{Shifted
       convolution sums of coefficients of symmetric power $L$-functions with
       $k$-full kernels over sums of squares in arithmetic progressions}
\author{Jewel Mahajan}
\address{(Jewel Mahajan) Department of Mathematical Sciences, Indian Institute of Science Education and Research (IISER) Berhampur, Ganjam, Odisha 760003, India}
\email{jewelmahajan421@gmail.com}
\author{Arnab Mitra}
\address{(Arnab Mitra) School of Mathematical Sciences, National Institute of Science Education and Research,  An OCC of Homi Bhabha National Institute, Bhubaneswar, Via: Jatni, Khurda, Odisha-752050, India}
\email{arnab.mitra@niser.ac.in}
\date{June 2026}

\begin{document}
\begin{abstract}
  Let $q$ be an integer and let $f$ be a normalised Hecke eigenform of
  integral weight for the full modular group. Let $L(s,\mathrm{sym}^j f)$ denote
  the $j$-th symmetric power $L$-function associated to $f$, and let
  $\lambda_{\mathrm{sym}^j f}(n)$ denote its $n$-th coefficient. We study the
  behaviour of the partial sum of $\lambda_{\mathrm{sym}^j f}(n)$, and of its
  second moment, taken over those sums of $m$ squares that are congruent
  to $1$ modulo $q$. As an application, we investigate the shifted
  convolution sum of $\lambda_{\mathrm{sym}^j f}(n)$ against a $k$-full kernel
  function, for any $k \geq 2$. We also study the number of sign changes
  of $\lambda_{\mathrm{sym}^j f}(n)$ twisted with a $k$-full kernel function,
  again over sums of $m$ squares. Throughout, $m$ is even with
  $m \in \{2,4,6,8,10,12\}$.
\end{abstract}

\maketitle

\section{Introduction}
Let $k \geq 2$ be any integer. An integer valued function $a(n)$ is called a $k$-full function if $p^k | a(n)$ whenever $p|a(n)$. Note that any positive integer $n  1$ can be uniquely decomposed into $n = a_1(n)a_2(n)$ with $(a_1(n), a_2(n)) = 1$, where $a_1(n)$ is $k$-free and $a_2(n)$ is $k$-full. A non-negative integer valued function $a(n)$ is called a $k$-full kernel function if $a(n) = a(a_2(n)) $ for all $n \geq 1$ and $a(n) \ll n^\epsilon$ for all $\epsilon > 0$. The notion of a $k$-full kernel function is given by Ivi\'{c} and Tenenbaum \cite{Tenenbaum1986}. Note that $k$-full kernel functions are not necessarily multiplicative.

Let $S_{\kappa}$ be the space of all holomorphic cusp forms of integer weight $\kappa$ for the full modular group $\mathrm{SL}(2,\mathbb{Z})$ and $f \in S_{\kappa}$. Let $\lambda_{f}(n)$ be the normalised $n$-th Fourier coefficient of the Fourier expansion of $f(z)$ at the cusp $\infty$, i.e.,
\[
f(z)=\sum_{n=1}^{\infty}\lambda_{f}(n)n^{\frac{\kappa-1}{2}}e^{2\pi inz},
\]
where $\Im(z)>0$. Then the $L$-function attached to $\lambda_{f}(n)$ is defined as
\[
L(s,f)=\sum_{n=1}^{\infty}\frac{\lambda_{f}(n)}{n^{s}}
\]
for $\Re(s)>1$, where $\lambda_{f}(n)$ are the eigenvalues of all the Hecke operators $T_{n}$.

In 1974, Deligne \cite{deligne1974} proved that for any prime $p$, there exist complex numbers $\alpha(p)$ and $\beta(p)$ such that
\begin{align}\label{eq:deligne1}
\alpha(p)\beta(p) &= 1,  
\end{align}
\begin{align}\label{eq:deligne2}
|\alpha(p)| = |\beta(p)| &= 1. 
\end{align}
Then $L(s,f)$ can be written as
\[
L(s,f)=\prod_{p}\left(1-\frac{\alpha(p)}{p^{s}}\right)^{-1}\left(1-\frac{\beta(p)}{p^{s}}\right)^{-1},
\]
where the product runs over all primes \(p\). Also, $|\lambda_{f}(n)|\leq d(n)$, where $d(n)$ is the divisor function.

The symmetric square $L$-function is defined as
\[
L(s,\mathrm{sym}^{2}f):=\sum_{n=1}^{\infty}\frac{\lambda_{\mathrm{sym}^{2}f}(n)}{n^{s}}=\prod_{p}\left(1-\frac{\alpha^{2}(p)}{p^{s}}\right)^{-1}\left(1-\frac{\beta^{2}(p)}{p^{s}}\right)^{-1}\left(1-\frac{1}{p^{s}}\right)^{-1}
\]
for $\Re(s)>1$, where $\lambda_{\mathrm{sym}^{2}f}(n)$ is multiplicative.

In $1987$, Erd\H{o}s and Ivi\'{c} \cite{Ivic2} proved that \begin{align*}
    \sum_{n\leq x } a(n)d(n+1) = C_1x\log x + C_2x + O(x^{\frac{8}{9}+\epsilon})
\end{align*}
\begin{align*}
    \sum_{n\leq x } a(n)w(n+1) = D_1x\log \log x + D_2x + O\left (\frac{x}{\log x} \right ),
\end{align*}
where $d(n)$ is the divisor function, and $w(n)$ is the number of different prime factors of $n$, and $C_1>0$, $D_1 > 0$, $C_2$, $D_2$ are constants that can be evaluated explicitly.

In \cite{WangD}, L\"{u} and Wang investigated the shifted convolution
sums of squares of Fourier coefficients with a $2$-full kernel function $a(n)$ and obtained an asymptotic
formula for the sum
\begin{align*}
    \sum_{n \leq x}a(n)\lambda^2_f(n+1).
\end{align*}

Later, Venkatasubbareddy and Sankaranarayanan \cite{Venkat} extended the problem to fourth moments, and generalised to $k$-full kernel functions. For $q \geq 100$ and $q \ll x^{\frac{23}{181}-\epsilon}$, they proved that 
\begin{align}
    \sum_{\substack{n \leq x+1 \\ n \equiv 1 ~( \mathrm{mod }~q)}} \lambda_f^4(n) = c_1x \log x \frac{\phi(q)}{q^2}+ O\left( \frac{x^{\frac{158}{181}+\epsilon}q^{1+\epsilon}}{\phi(q)}\right)
\end{align}

and 

for $q \geq 100$ and $q \ll x^{\frac{3}{23}-\epsilon}$, they proved that 
\begin{align}
    \sum_{\substack{n \leq x+1 \\ n \equiv 1 ~( \mathrm{mod }~q)}} \lambda_f^4(n) = c_1x \log x \frac{\phi(q)}{q^2}+ O\left( \frac{x^{\frac{20}{23}+\epsilon}q^{1+\epsilon}}{\phi(q)}\right).
\end{align}
As an application, for a $k$-full kernel function $a(n)$, they proved that

\begin{align}
    \sum_{n \leq x}a(n)\lambda_f^4(n+1) = c_2x \log x+O\left(x^{\frac{520k+23}{543k}+\epsilon} \right).
\end{align}

On the other hand, Wang \cite{Youjun2024} proved that for a $k$-full kernel function $a(n)$ and $f \in S_{\kappa}$, 
\begin{align}
    \sum_{n\leq x } a(n)\lambda^2_{\mathrm{sym}^jf}(n+1) = C_{f,j}x+O \left( x^{1-\frac{2k-2}{3k(j+1)^2}+\epsilon} \right),
\end{align}
and \begin{align}
    \sum_{n\leq x } a(n)\lambda^3_{\mathrm{sym}^2f}(n+1) = Dx+O\left( x^{\frac{1805k+46}{1851k}+\epsilon} \right).
\end{align}

This article investigates the behaviour of the partial sum of $\lambda_{\mathrm{sym}^{j}f}(n)$ and its second moment for a sequence of integers $n$ which are expressible as a sum of $m$ squares, which are $1$ modulo $q$, where $q$ is an integer. As an application, we examine their shifted convolution sum with a weight function, which is a $k$-full kernel function for $k \geq 2$. We also examine the number of sign changes of $a(n)\lambda_{\mathrm{sym}^jf}(n+1)$, where $n+1$ runs over the sum of $m$ squares. Here $m = 2,4,6,8,10,12$.

In particular, we prove 

\begin{theorem} \label{thm:main}
    Let $f \in S_{\kappa}$ and let $q \geq 100$ be any integer. Then for any $\epsilon > 0$ and $q \ll x^{\frac{1}{j+1}-\epsilon}$ , where $j \geq 2$ is a fixed integer, we have $$\sum_{\substack{a_1^2+a_2^2 \leq x+1 \\ a_1^2+a_2^2 \equiv 1 ~( \mathrm{mod }~q) \\ (a_1,a_2) \in \mathbb{Z}^2}}\lambda_{\mathrm{sym}^jf}(a_{1}^{2}+a_2^2) = O \left( \frac{x^{1+\epsilon-\frac{1}{j+1}}q^{1+\epsilon}}{\phi(q)}\right).$$
\end{theorem}

\begin{theorem} \label{thm:main(4681012)}
    Let $f \in S_{\kappa}$ and let $q \geq 100$ be any integer. Then for any $\epsilon > 0$ and $q \ll x^{\frac{2}{j+1}-\epsilon}$ , where $j \geq 2$ is a fixed integer and for $m\in \{4,6,8,10,12\}$, we have $$\sum_{\substack{a_{1}^{2}+\cdots+a_{m}^{2} \leq x+1 \\ a_{1}^{2}+\cdots+a_{m}^{2} \equiv 1 ~( \mathrm{mod }~q) \\ (a_{1},\ldots,a_{m})\in\mathbb{Z}^{m}}}\lambda_{\mathrm{sym}^jf}(a_{1}^{2}+\cdots+a_{m}^{2}) = O\left( \frac{x^{\frac{m}{2}+\epsilon-\frac{2}{j+3}}q^{\frac{j+1}{j+3}+\epsilon}}{\phi(q)} \right).$$
\end{theorem}

\begin{theorem} \label{thm:main1}
    Let $f \in S_{\kappa}$ and let $q \geq 100$ be any integer. Then for any $\epsilon > 0$ and $q \ll x^{\frac{1}{(j+1)^2}-\epsilon}$, where $j \geq 2$ is a fixed integer, we have $$\sum_{\substack{a_1^2+a_2^2 \leq x+1 \\ a_1^2+a_2^2 \equiv 1 ~( \mathrm{mod }~q) \\ (a_1,a_2) \in \mathbb{Z}^2}}\lambda_{\mathrm{sym}^jf}^2(a_{1}^{2}+a_2^2) =\frac{c_{j,f}}{q}x+ O \left( \frac{x^{1+\epsilon-\frac{1}{(j+1)^2}}q^{1+\epsilon}}{\phi(q)}\right),$$  where $c_{j,f}$ is a constant that depends on $j$ and $f$.  
\end{theorem}

\begin{theorem} \label{thm:main1(468)}
    Let $f \in S_{\kappa}$ and let $q \geq 100$ be any integer. Then for any $\epsilon > 0$ and $q \ll x^{\frac{2}{(j+1)^2}-\epsilon}$, where $j \geq 2$ is a fixed integer and for $m\in \{4,6,8,10,12\}$, we have $$\sum_{\substack{a_{1}^{2}+\cdots+a_{m}^{2} \leq x \\ a_{1}^{2}+\cdots+a_{m}^{2} \equiv 1 ~( \mathrm{mod }~q) \\ (a_{1},\ldots,a_{m})\in\mathbb{Z}^{m}}}\lambda_{\mathrm{sym}^jf}^2(a_{1}^{2}+\cdots+a_{m}^{2}) = \frac{c_{m,j,f}}{q}x^\frac{m}{2} + O\left( \frac{x^{\frac{m}{2}+\epsilon-\frac{2}{(j+1)^2}}q^{1+\epsilon}}{\phi(q)} \right),$$ where $c_{m,j,f}$ is a constant that depends on $j$, $f$ and $m$.
\end{theorem}

\begin{theorem} \label{thm:main2}
    For any integer $k \geq 2$, let $a(n)$ be a $k$-full kernel function and $f \in S_{\kappa}$. Let $j \geq 2$ be an integer. Then for any $\epsilon >0$ and $m\in \{4,6,8,10,12\}$, we have
    $$\sum_{\substack{n \leq x \\ (a_{1},\ldots,a_{m})\in\mathbb{Z}^{m} \\ n = a_{1}^{2}+\cdots+a_{m}^{2}-1}}a(n)\lambda_{\mathrm{sym}^jf}(n+1) =  O\left( x^{\frac{m}{2}-\frac{2k-2}{3k(j+1)}+\epsilon} \right),$$
    and also $$\sum_{\substack{n \leq x+1 \\ (a_{1},a_{2})\in\mathbb{Z}^{2} \\ n = a_{1}^{2}+a_{2}^{2}-1}}a(n)\lambda_{\mathrm{sym}^jf}(n+1) =  O\left( x^{1-\frac{k-1}{3k(j+1)}+\epsilon} \right).$$
\end{theorem}

\begin{theorem} \label{thm:main3}
    For any integer $k \geq 2$, let $a(n)$ be a $k$-full kernel function and $f \in S_{\kappa}$. Let $j \geq 2$ be an integer. Then for any $\epsilon >0$ and $m\in \{4,6,8,10,12\}$, we have
    $$\sum_{\substack{n \leq x+1 \\ (a_{1},\ldots,a_{m})\in\mathbb{Z}^{m} \\ n = a_{1}^{2}+\cdots+a_{m}^{2}-1}}a^2(n)\lambda_{\mathrm{sym}^jf}^2(n+1) = D_{j,f,m}x^\frac{m}{2}+ O\left( x^{\frac{m}{2}-\frac{2k-2}{3(j+1)^2k}+\epsilon}\right),$$

    and also $$\sum_{\substack{n \leq x+1 \\ (a_{1},a_{2})\in\mathbb{Z}^{2} \\ n = a_{1}^{2}+a_{2}^{2}-1}}a^2(n)\lambda_{\mathrm{sym}^jf}^2(n+1) = D_{j,f,2}x+ O\left( x^{1-\frac{k-1}{3(j+1)^2k}+\epsilon}\right),$$

    where $D_{j,f,m}$ and $D_{j,f,2}$ are constants that can be evaluated explicitly.
\end{theorem}

\begin{theorem} \label{thm:mainsign}
    Let $f \in S_{\kappa}$ and $j \geq 2$ be a fixed integer and $a(n)$ be a $k$-full kernel function for any integer $k \geq 2$. Then, for sufficiently large $x$, the sequence $$\{ a(n)\lambda_{\mathrm{sym}^jf}(n+1) ~|~ n = a_1^2+a_2^2-1 , a_i \in \mathbb{Z} \}$$ has at least $x^{1-\delta_j}$ sign changes between $x$ and $2x$, for any $\delta_j $ with $1-\frac{k-1}{3k(j+1)^2}  < \delta_j < 1$.
\end{theorem}

\begin{theorem} \label{thm:mainsign1}
        Let $f \in S_{\kappa}$ and $j \geq 2$ be a fixed integer, $ m \in \{4,6,8, 10, 12\}$ and $a(n)$ be a $k$-full kernel function for any integer $k \geq 2$. Then for sufficiently large $x$, the sequence
        $$\{ a(n)\lambda_{\mathrm{sym}^jf}(n+1) ~|~ n = \sum_{i=1}^{m}a_i^2-1 , a_i \in \mathbb{Z} \}$$ has at least $x^{1-\delta_j}$ sign changes between $x$ and $2x$, for any $\delta_j $ with $1-\frac{2k-2}{3k(j+1)^2}  < \delta_j < 1-\frac{k-1}{3k(j+1)^2}$.
\end{theorem}

\subsection*{Organisation of the article}
This article is organised as follows. In Section~\ref{sec:Prel}, we discuss the preliminaries, important lemmas, and bounds necessary to prove our results. The proofs of our main results, related to the partial sums of the symmetric power $L$-function attached to Hecke eigenforms, are given in Sections~\ref{sec:main},~\ref{sec:main(4681012)},~\ref{sec:main1}, and~\ref{sec:main1(468)}. Sections~\ref{sec:conv-sum} and~\ref{sec:conv-sum-square} deal with results related to shifted convolution sums of the symmetric power $L$-function twisted with $k$-full kernel functions. Finally, Sections~\ref{sec:sign-change-2} and~\ref{sec:sign-change-m} contain the results related to the number of sign changes of $\lambda_{\mathrm{sym}^jf}(n)$ twisted with a $k$-full kernel function over the sum of squares.

\section{Preliminaries and some important lemmas} \label{sec:Prel}
Let
\[
r_k(n) := \#\bigl\{(n_1, n_2, \dots, n_k) \in \mathbb{Z}^k : n_1^2 + n_2^2 + \cdots + n_k^2 = n\bigr\},
\]
where we count all ordered $k$-tuples of integers $(n_1,\dots,n_k)$ satisfying the equation, 
including zeros and treating different signs and permutations as distinct. 

We will define the functions $r_m(n)$, where $m=2,4,6,8,10,12$, which are defined as follows
\begin{defn} \cite[p.~121]{Grosswald}
    For any positive integer $n$, define \begin{align} \label{eq:r_2(n)}
        r_2(n) = 4\sum_{d|n}\rchi_ 4(d), 
        \end{align}
    \begin{align} \label{eq:r_4(n)}
        r_4(n) = 8\sum_{d|n}d,
    \end{align}
    \begin{align} \label{eq:r_6(n)}
        r_6(n) = 16\sum_{d|n}d^2\rchi_ 4\left(\frac{n}{d} \right)-4\sum_{d|n}d^2\rchi_ 4(d),
    \end{align}
    \begin{align} \label{eq:r_8(n)}
        r_8(n) = 16\sum_{d|n}(-1)^{n+d}d^3,
    \end{align}
    \begin{align} \label{eq:r_10}
        r_{10}(n)=\frac{64}{5}\biggl\{\sum_{d|n}\rchi(d^{\prime})d^{4} + \frac{1}{16}\sum_{d|n}\rchi(d)d^{4}\biggr\} + \frac{32}{5}a_n,
    \end{align}
    \begin{align} \label{eq:r_12}
        r_{12}(n)= 8\sum_{d|n}(-1)^{n+d+\frac{n}{d}-1}d^5+16b_n,
    \end{align}

where $\rchi_ 4$ is the non-principal Dirichlet character modulo $4$, i.e.,
\[
\rchi_ 4(n)=\begin{cases}
1 & \text{if } n\equiv 1  \pmod{4}, \\
-1 & \text{if } n\equiv -1  \pmod{4}, \\
0 & \text{if } n\equiv 0   \pmod{2},
\end{cases}
\]
$a_n$ is defined via the identity
\[
\theta_2^4 \, \theta_3^2 \, \theta_4^4 = 16 \sum_{n=1}^\infty a_n q^n \qquad (q = e^{2\pi i z}),
\]
where the classical theta functions are given by
$$\theta_2= 2q^{\frac{1}{4}}\prod_{m=1}^{\infty}(1-q^{2m})(1+q^{2m})^2,$$
$$\theta_3= \prod_{m=1}^{\infty}(1-q^{2m})(1+q^{2m-1})^2,\text{ and}$$
$$\theta_4= \prod_{m=1}^{\infty}(1-q^{2m})(1-q^{2m-1})^2 \quad (|q|<1).$$ 

and $b_n$ is defined via the identity
\[
\left (\frac{\theta_1'}{\pi}\right)^4 = 16 \sum_{n=1}^\infty b_n q^n \qquad (q = e^{2\pi i z}),
\]
where 
$$ \theta_1' = 2\pi q^{1/4}\prod_{m=1}^{\infty}(1-q^{2m})^3 \quad (|q| < 1).$$
\end{defn}

Define the arithmetic functions
\begin{defn}
    \begin{align} \label{eq:l_1(n)}
        l_1(n) = \sum_{d|n}\rchi_ 4(d),
    \end{align}
    \begin{align} \label{eq:l_2(n)}
        l_2(n) = \sum_{d|n}d,
    \end{align}
    \begin{align} \label{eq:l_3(n)&v_3(n)}
        l_3(n) = \sum_{d|n}d^2\rchi_ 4\left( \frac{n}{d} \right), \quad v_3(n) \sum_{d|n} d^2 \rchi_ 4(d),
    \end{align}
    \begin{align} \label{eq:l_4(n)}
        l_4(n) = \sum_{d|n}(-1)^{n+d}d^3.
    \end{align}
    \begin{align} \label{eq:l_5(n)&v_5(n)}
        l_5(n) = \sum_{d \mid n} \rchi\!\left(\frac{n}{d}\right) d^{\,4},
        \qquad
        v_5(n) = \sum_{d \mid n} \rchi(d)\,d^{\,4}, 
    \end{align}
    \begin{align}\label{eq:l_6(n)}
        l_6(n) = \sum_{d|n}(-1)^{n+d+\frac{n}{d}-1}d^5.
    \end{align}
\end{defn}

Observing the definitions of the arithmetic functions, we have 
\begin{align} \label{eq:r_2&l_1}
    r_2(n) = 4l_1(n) \ll n^\epsilon,
\end{align}
\begin{align} \label{eq:r_4&l_2}
    r_4(n) = 8l_2(n) \ll n^{1+\epsilon},
\end{align}
\begin{align} \label{eq:r_6&l_3&v_3}
    r_6(n) = 16l_3(n)-4v_3(n) \ll n^{2+\epsilon},
\end{align}
\begin{align} \label{eq:r_8&l_4}
    r_8(n) = 16l_4(n) \ll n^{3+\epsilon},
\end{align}
\begin{align}\label{eq:r_10&l_5&v_5}
    r_{10}(n) = \frac{64}{5}\,l_5(n) + \frac{4}{5}\,v_5(n) + \frac{32}{5}\,a_n, 
\end{align}
\begin{align} \label{eq:r_12&l_6}
    r_{12}(n) = 8l_6(n)+16b_n.
\end{align}
where $\epsilon > 0$ and we note that in general, $r_{m}(n) \ll n^{\frac{m}{2}-1}+\epsilon$ for $m=2,4,6,8,10,12$.

We have \begin{align}
    \sum_{\substack{a_{1}^{2}+\cdots+a_{m}^{2}\leq x \\ (a_{1},\ldots,a_{m})\in\mathbb{Z}^{m}}}\lambda_{\mathrm{sym}^{j}f}(\sum_{i=1}^{m}a_{i}^{2})\rchi(\sum_{i=1}^{m}a_{i}^{2})
    &= \sum_{n\leq x}\lambda_{\mathrm{sym}^{j}f}(n)\rchi(n)\sum_{\substack{n=a_{1}^{2}+\cdots+a_{m}^{2} \\ (a_{1},\ldots,a_{m})\in\mathbb{Z}^{10}}} 1 \\
    &= \sum_{n\leq x}\lambda_{\mathrm{sym}^{j}f}(n)\rchi(n)r_{m}(n)
\end{align}

Now using \eqref{eq:r_2&l_1}, \eqref{eq:r_4&l_2}, \eqref{eq:r_6&l_3&v_3},  \eqref{eq:r_8&l_4}, \eqref{eq:r_10&l_5&v_5}, \eqref{eq:r_12&l_6} and above equation, we have \begin{align} \label{eq:Lmbda2Def}
    \sum_{n\leq x}\lambda_{\mathrm{sym}^{j}f}(n)\rchi(n)r_{2}(n) = 4\sum_{n\leq x}\lambda_{\mathrm{sym}^{j}f}(n)\rchi(n)l_1(n),
\end{align}
\begin{align} \label{eq:Lmbda4Def}
    \sum_{n\leq x}\lambda_{\mathrm{sym}^{j}f}(n)\rchi(n)r_{4}(n) = 8\sum_{n\leq x}\lambda_{\mathrm{sym}^{j}f}(n)\rchi(n)l_2(n),
\end{align}
\begin{align} \label{eq:Lmbda6Def}
    \sum_{n\leq x}\lambda_{\mathrm{sym}^{j}f}(n)\rchi(n)r_{6}(n) = 16\sum_{n\leq x}\lambda_{\mathrm{sym}^{j}f}(n)\rchi(n)l_3(n)-4\sum_{n\leq x}\lambda_{\mathrm{sym}^{j}f}(n)\rchi(n)v_3(n),
\end{align}
\begin{align} \label{eq:Lmbda8Def}
    \sum_{n\leq x}\lambda_{\mathrm{sym}^{j}f}(n)\rchi(n)r_{8}(n) = 16\sum_{n\leq x}\lambda_{\mathrm{sym}^{j}f}(n)\rchi(n)l_4(n),
\end{align}
\begin{align} \label{eq:Lmbda10Def}
    \sum_{\substack{n\leq x }}\lambda_{\mathrm{sym}^{j}f}(n)\rchi(n)r_{10}(n) &= \frac{64}{5}\sum_{n\leq x}\lambda_{\mathrm{sym}^{j}f}(n)\rchi(n)l_5(n)+\frac{4}{5}\sum_{n\leq x}\lambda_{\mathrm{sym}^{j}f}(n)\rchi(n)v(n)\\
    &+\frac{32}{5}\sum _{n\leq x}\lambda_{\mathrm{sym}^{j}f}(n)\rchi(n)a_n,
\end{align}
\begin{align}\label{eq:Lmbda12Def}
    \sum_{\substack{n\leq x }}\lambda_{\mathrm{sym}^{j}f}(n)\rchi(n)r_{12}(n) &= 8\sum _{n\leq x}\lambda_{\mathrm{sym}^{j}f}(n)\rchi(n)l_{6}(n)+16\sum _{n\leq x}\lambda_{\mathrm{sym}^{j}f}(n)\rchi(n)b_n.
\end{align}

\begin{subsection}{Bounds for \texorpdfstring{$\lambda_{\mathrm{sym}^jf}(n)$}{}}
Note that \eqref{eq:deligne2} yields  $\left|1-\frac{\alpha^{j-1}(p)\beta^i(p)}{p^s}\right| \geq 1-\frac{1}{p^\sigma}>0$  for $\Re (s)=\sigma>1$. Therefore,
\begin{align*}
        |L(s,\mathrm{sym}^jf)| &\leq \prod_{p}\prod_{i=0}^{j}\left(1-\frac{1}{p^\sigma} \right)^{-1} 
        =\prod_{i=0}^{j}\zeta(\sigma) = \zeta(\sigma)^{j+1}
        = \sum_{n=1}^{\infty}\frac{d_{j+1}(n)}{n^\sigma},
\end{align*} 
where $d_{j+1}(n)$ is the number of ways of expressing $n$ as a product of $j+1$ factors. 
Since $d_{k}(n) \le d(n)^{k-1}$ for positive integers $k$ and $n$, and since $d(n) \ll_{\epsilon} n^{\epsilon}$ for any $\epsilon > 0$, we obtain
\[
d_{k}(n) \ll_{k,\epsilon} n^{\epsilon} \quad \text{for any } \epsilon > 0.
\]
Therefore, the Dirichlet series for $L(s, \operatorname{sym}^j f)$ is absolutely convergent for $\Re(s) > 1$.

Note that \eqref{eq:deligne1} and \eqref{eq:deligne2} imply
\[
|\lambda_{\operatorname{sym}^{j}f}(n)| \le d_{j+1}(n).
\]
Consequently, for any $\epsilon>0$,
\begin{equation} \label{eq:lambdaSymbound}
|\lambda_{\operatorname{sym}^{j}f}(n)| \ll_{j,\epsilon} n^{\epsilon}.
\end{equation}

Since $\lambda_{\operatorname{sym}^{j}f}(n)$ is multiplicative, $L(s,\operatorname{sym}^{j}f)$ admits an Euler product
\begin{equation} \label{eq:eulerproduct}
L(s,\operatorname{sym}^{j}f) = \prod_{p} \Bigl( 1 + \frac{\lambda_{\operatorname{sym}^{j}f}(p)}{p^{s}} 
    + \frac{\lambda_{\operatorname{sym}^{j}f}(p^{2})}{p^{2s}} + \dots \Bigr),
\end{equation}
absolutely convergent for $\Re(s) > 1$.


Observe that
\begin{align} \label{eq:lambdaatp}
\lambda_{\mathrm{sym}^{j}f}(p)=\sum_{m=0}^{j}\alpha^{j-m}(p)\beta^{m}(p).
\end{align}

Moreover, Hecke theory gives the relation
\begin{equation} \label{eq:hecke}
\lambda_{\operatorname{sym}^{j}f}(p) = \lambda_f(p^{j}) \qquad (\text{$p$ prime}),
\end{equation}
and we also know that \begin{align} \label{lambdasquare}
    \lambda_{\mathrm{sym}^{j}f}^{2}(p) \;=\; 1 \;+\; \sum_{\ell=1}^{j} \lambda_{\mathrm{sym}^{2\ell}f}(p).
\end{align}

\end{subsection}

\begin{subsection}{Important lemmas}
\begin{lemma} \label{lem:F1}
Let \(f\) be a normalised primitive holomorphic cusp form of weight \(k\) for \(\operatorname{SL}(2,\mathbb{Z})\), 
and let \(\lambda_{\operatorname{sym}^{j}f}(n)\) denote the \(n\)-th normalised Fourier coefficient 
of the \(j\)-th symmetric power \(L\)-function attached to \(f\) and let $\rchi$ be a Dirichlet character modulo $q$.  

Define
\[
F_{j}^{(1)}(s, \rchi)=\sum_{n=1}^{\infty}\frac{\lambda_{\operatorname{sym}^{j}f}(n)\,l_1(n)\rchi(n)}{n^{s}}, \qquad \Re(s)>1,
\]
where \(l_1(n)\) is given by \eqref{eq:l_1(n)}. Then \(F_{j}^{(1)}(s,\rchi)\) admits a factorisation
\[
F_{j}^{(1)}(s, \rchi)=G_{j}^{(1)}(s, \rchi)\,H_{j}^{(1)}(s, \rchi),
\]
in which
\[
G_{j}^{(1)}(s,\rchi):=L\!\bigl(s,\operatorname{sym}^{j}f \otimes \rchi\bigr)\; 
        L\!\bigl(s,\operatorname{sym}^{j}f\otimes\rchi_ 4\rchi\bigr),
\]
\(\rchi_ 4\) is the unique non-principal Dirichlet character modulo \(4\), 
and \(H_{j}^{(1)}(s, \rchi)\) is a Dirichlet series that converges absolutely and uniformly 
in the half-plane \(\Re(s)>\frac{1}{2}\).
\end{lemma}

\begin{proof}
We know that,
\[
\lambda_{\operatorname{sym}^{j}f}(n)\,l_1(n)\rchi(n) \ll n^{\epsilon} \qquad (\epsilon > 0),
\]
which implies that the Dirichlet series \(F_j^{(1)}(s,\rchi)\) converges absolutely for \(\Re(s) > 1\). Since \(\lambda_{\operatorname{sym}^{j}f}(n)\) is multiplicative, \(F_j^1(s)\) therefore admits an Euler product in this half-plane in \(\Re(s) > 1\):
\begin{align} \label{eq:FjEuler1}
    F_{j}(s,\rchi)=\prod_{p}
        \Bigl(1+
        \frac{\lambda_{\operatorname{sym}^{j}f}(p)\,l_1(p)\rchi(p)}{p^{s}}+
        \frac{\lambda_{\operatorname{sym}^{j}f}(p^{2})\,l_1(p^{2})\rchi(p^2)}{p^{2s}}+
        \cdots
        +\frac{\lambda_{\operatorname{sym}^{j}f}(p^{m})\,l_(p^{m})\rchi(p^m)}{p^{ms}}+
        \cdots\Bigr).
\end{align}

Now define the multiplicative function \(b_1(n)\) via its Euler product
\[
\sum_{n=1}^{\infty}\frac{b_1(n)}{n^{s}} 
    := L(s,\operatorname{sym}^{j}f\otimes \rchi)\;L(s,\operatorname{sym}^{j}f\otimes\rchi_ 4\rchi)
    (=: G_{j}(s, \rchi)),
\]
so that in particular for primes \(p\),
\[
b_1(p) = \lambda_{\operatorname{sym}^{j}f}(p)\,\rchi(p) 
       + \lambda_{\operatorname{sym}^{j}f}(p)\,\rchi_ 4\rchi(p).
\]
Since \(l(p) = 1+\rchi_ 4(p)\), we obtain that $b_1(p) = \lambda_{\operatorname{sym}^{j}f}(p)\,l_1(p)\rchi_ (p)$, establishing the desired equality at each prime.

        But note that $b_1(p^k) \neq \lambda_{\mathrm{sym}^jf}(p^k)l(p^k)\rchi(p^k) \text{ for } k>1$ and again 
        \begin{align*}
            |b_1(n)| 
            =|(\lambda_{\mathrm{sym}^jf}\rchi*\lambda_{\mathrm{sym}^jf}\rchi_ {4}\rchi)(n)|
            &\leq \sum_{d|n}|\lambda_{\mathrm{sym}^jf}(d)\rchi(d)||\lambda_{\mathrm{sym}^jf}\left(\frac{n}{d}\right)\rchi_ {4}\rchi\left(\frac{n}{d}\right)| \\
            & \leq \sum_{d|n}d^{\epsilon}\left(\frac{n}{d}\right)^\epsilon \leq n^{\epsilon}d(n) \ll_\epsilon n^{\epsilon} \text{ for any } \epsilon >0.
        \end{align*}
        So $\displaystyle\sum_{n=1}^{\infty}\frac{b_1(n)}{n^s}$ is absolutely convergent by $\Re(s) > 1$ and the Euler product ensures that $$\sum_{n=1}^{\infty}\frac{b_1(n)}{n^s} = \prod_p \left(1+\sum_{m \geq1}\frac{b_1(p^m)}{p^{ms}}\right)\quad (\Re (s)>1).$$
        Now, \begin{align*}
            \bigg|\sum_{m=1}^{\infty}\frac{b_1(p^m)}{p^{ms}}\bigg| &\leq \sum_{m=1}^{\infty}\frac{p^{\epsilon m}}{p^{m\sigma}} \leq \sum_{m=1}^{\infty}\frac{p^{\epsilon m}}{p^{(1+2\epsilon)m}} = \sum_{m=1}^{\infty} \frac{1}{p^{m(1+\epsilon)}} = \frac{1}{p^{1+\epsilon}-1} <1 
        \end{align*}
        for $\Re(s) > 1+2\epsilon$.
        
        Let $$A = \sum_{m=1}^{\infty} \frac{\lambda_{\mathrm{sym}^jf}(p^m)l_1(p^m)\rchi(p^m)}{p^{ms}}, \quad \text{and}$$
        $$ B = \sum_{m=1}^{\infty}\frac{b_1(p^m)}{p^{ms}} \quad (|B| < 1 ).$$

        Therefore, \begin{align*}
            \frac{1+A}{1+B} &= (1+A)(1-B+B^2-\cdots) \\&= 1+A -B -AB + \cdots \\
            &= 1+ \frac{\lambda_{\mathrm{sym}^jf}(p^2)l_1(p^2)\rchi(p^2)-b_1(p^2)}{p^{2s}}+ \cdots+\frac{c(p^m)}{p^{ms}}+ \cdots\\
            &= \sum_{n \geq 1} \frac{c_p(n)}{n^s} \quad \text{(say)},
        \end{align*}
        where $$c_p(n) = \begin{cases}
        1 & \text{if } n=1, \\
        c(n) & \text{if } n = p^m \,(m \geq 2), \\
        0 & \text{otherwise.}
        \end{cases}$$
        Note that the above equality holds for $\Re(s) > 1+2\epsilon$ for all $\epsilon > 0$, and that the series is absolutely convergent in this region.
        Also note that $c_p(n) \ll n^\epsilon$ for all $\epsilon > 0$.
        We define $c(n)$ for any $n \in \mathbb{N}$ by $$\prod_{p}\frac{1+A}{1+B} = \prod_p\left( 1+ \sum_{m \geq 1}\frac{c(p^m)}{p^{ms}}\right) = \sum_{n=1}^{\infty}\frac{c(n)}{n^s}.$$
        By construction, $c(n)$ is multiplicative.

         Now Define \begin{align*}
             H_j^{(1)}(s, \rchi) : = \frac{F_j^{(1)}(s,\rchi)}{G_j^{(1)}(s,\rchi)} &=\prod_{p} \frac{1+ \sum_{m \geq 1}\frac{\lambda_{\mathrm{sym}^jf}(p^m)l_1(p^m)\rchi(p^m)}{p^{ms}}}{1+\sum_{m\geq1}\frac{b(p^m)}{p^{ms}}} \\
             &= \prod_p \frac{1+A}{1+B} = \sum_{n=1}^{\infty}\frac{c(n)}{n^s}.
         \end{align*}

         We now show the region of convergence of $H_j^{(1)}(s,\rchi)$:
         \bigskip
         Note that \begin{align*}
             \sum_{m \geq 3}\bigg|\frac{c(p^m)}{p^{ms}}\bigg| &\leq \sum_{m\geq 3}\frac{p^{m\epsilon}}{p^{m\sigma}} = \sum_{m \geq 3} \frac{1}{p^{m(\sigma - \epsilon)}} \\
             &= \frac{1}{p^{2(\sigma-\epsilon)}(p^{\sigma -\epsilon}-1)} < \frac{1}{p^{2(\sigma-\epsilon)}} \text{ for } \epsilon >0, \text{ as small as possible}.
         \end{align*} The above inequality of the series is true for $\Re(s)> 1+\epsilon$.
         Again,\begin{align*}
             \frac{c(p^2)}{p^{2\sigma}} &= \frac{\lambda_{\mathrm{sym}^jf}(p^2)l_1(p^2)\rchi(p^2) - b(p^2)}{p^{2\sigma}} \\ 
             &= O(\frac{p^{2\epsilon}}{p^{2\sigma}}) = O(\frac{1}{p^{2\sigma-2\epsilon}}).
         \end{align*}

         Now $\prod_p\left(1 + \bigg|\frac{c(p^2)}{p^{2s}}\bigg|+ \sum_{m\geq 3} \bigg|\frac{c(p^m)}{p^{ms}}\bigg|\right) = \prod_p\left(1+u_p\right)$ is convergent if and only if $\sum_pu_p$ is convergent, where $u_p = |\frac{c(p^2)}{p^{2s}}|+\sum_{m \geq 3}|\frac
         {c(p^m)}{p^{ms}}|$. Note that $$ \sum_pu_p \ll  \sum_p \frac{1}{p^{2\sigma -2\epsilon}} $$ is absolutely convergent for $2\sigma -2\epsilon > 1$ i.e., $\sigma > \frac{1}{2}+\epsilon$ for $\epsilon >0$ as small as possible. 

         So in this region $H_j^{(1)}(s,\rchi) \ll_\epsilon 1$ and $H_j^{(1)}(s,\rchi)$ is absolutely convergent in $\Re(s) > \frac{1}{2}$. 
    \end{proof}

Proof of the following lemmas follows the same process as above. So we don't provide the details of the proofs.

\begin{lemma}\label{lem:F2}
Let \(f\) be a normalised primitive holomorphic cusp form of weight \(k\) for \(\operatorname{SL}(2,\mathbb{Z})\), 
and let \(\lambda_{\operatorname{sym}^{j}f}(n)\) denote the \(n\)-th normalised Fourier coefficient 
of the \(j\)-th symmetric power \(L\)-function attached to \(f\) and let $\rchi$ be a Dirichlet character modulo $q$.  

Define
\[
F_{j}^{(2)}(s, \rchi)=\sum_{n=1}^{\infty}\frac{\lambda_{\operatorname{sym}^{j}f}(n)\,l_2(n)\rchi(n)}{n^{s}}, \qquad \Re(s)>2,
\]
where \(l_2(n)\) is given by \eqref{eq:l_2(n)}. Then \(F_{j}^{(2)}(s,\rchi)\) admits a factorisation
\[
F_{j}^{(2)}(s, \rchi)=G_{j}^{(2)}(s, \rchi)\,H_{j}^{(2)}(s, \rchi),
\]
in which
\[
G_{j}^{(2)}(s,\rchi):=L\!\bigl(s,\operatorname{sym}^{j}f \otimes \rchi\bigr)\; 
        L\!\bigl(s-1,\operatorname{sym}^{j}f\otimes\rchi\bigr)
\]
and \(H_{j}^{(2)}(s, \rchi)\) is a Dirichlet series that converges absolutely and uniformly 
in the half-plane \(\Re(s)>\frac{3}{2}\).
\end{lemma}

\begin{lemma} \label{lem:F3'}
Let \(f\) be a normalised primitive holomorphic cusp form of weight \(k\) for \(\operatorname{SL}(2,\mathbb{Z})\), 
and let \(\lambda_{\operatorname{sym}^{j}f}(n)\) denote the \(n\)-th normalised Fourier coefficient 
of the \(j\)-th symmetric power \(L\)-function attached to \(f\) and let $\rchi$ be a Dirichlet character modulo $q$.  

Define
\[
F_{j_1}^{(3)}(s, \rchi)=\sum_{n=1}^{\infty}\frac{\lambda_{\operatorname{sym}^{j}f}(n)\,l_3(n)\rchi(n)}{n^{s}}, \qquad \Re(s)>3,
\]
where \(l_3(n)\) is given by \eqref{eq:l_3(n)&v_3(n)}. Then \(F_{j_1}^{(3)}(s,\rchi)\) admits a factorisation
\[
F_{j_1}^{(3)}(s, \rchi)=G_{j_1}^{(3)}(s, \rchi)\,H_{j_1}^{(3)}(s, \rchi),
\]
in which
\[
G_{j_1}^{(3)}(s,\rchi):=L\!\bigl(s,\operatorname{sym}^{j}f \otimes \rchi_ 4\rchi\bigr)\; 
        L\!\bigl(s-2,\operatorname{sym}^{j}f\otimes\rchi\bigr)
\]
and \(H_{j_1}^{(3)}(s, \rchi)\) is a Dirichlet series that converges absolutely and uniformly 
in the half-plane \(\Re(s)>\frac{5}{2}\).
\end{lemma}

\begin{lemma} \label{lem:F3''}
Let \(f\) be a normalised primitive holomorphic cusp form of weight \(k\) for \(\operatorname{SL}(2,\mathbb{Z})\), 
and let \(\lambda_{\operatorname{sym}^{j}f}(n)\) denote the \(n\)-th normalised Fourier coefficient 
of the \(j\)-th symmetric power \(L\)-function attached to \(f\) and let $\rchi$ be a Dirichlet character modulo $q$.  

Define
\[
F_{j_2}^{(3)}(s, \rchi)=\sum_{n=1}^{\infty}\frac{\lambda_{\operatorname{sym}^{j}f}(n)\,v_3(n)\rchi(n)}{n^{s}}, \qquad \Re(s)>3,
\]
where \(v_3(n)\) is given by \eqref{eq:l_3(n)&v_3(n)}. Then \(F_{j_2}^{(3)}(s,\rchi)\) admits a factorisation
\[
F_{j_2}^{(3)}(s, \rchi)=G_{j_2}^{(3)}(s, \rchi)\,H_{j_2}^{(3)}(s, \rchi),
\]
in which
\[
G_{j_2}^{(3)}(s,\rchi):=L\!\bigl(s-2,\operatorname{sym}^{j}f \otimes \rchi_ 4\rchi\bigr)\; 
        L\!\bigl(s,\operatorname{sym}^{j}f\otimes\rchi\bigr)
\]
and \(H_{j_2}^{(3)}(s, \rchi)\) is a Dirichlet series converging absolutely and uniformly 
in the half-plane \(\Re(s)>\frac{5}{2}\).
\end{lemma}

\begin{lemma} \label{lem:F4}
Let \(f\) be a normalised primitive holomorphic cusp form of weight \(k\) for \(\operatorname{SL}(2,\mathbb{Z})\), 
and let \(\lambda_{\operatorname{sym}^{j}f}(n)\) denote the \(n\)-th normalised Fourier coefficient 
of the \(j\)-th symmetric power \(L\)-function attached to \(f\) and let $\rchi$ be a Dirichlet character modulo $q$.  

Define
\[
F_{j}^{(4)}(s, \rchi)=\sum_{n=1}^{\infty}\frac{\lambda_{\operatorname{sym}^{j}f}(n)\,l_4(n)\rchi(n)}{n^{s}}, \qquad \Re(s)>4,
\]
where \(l_4(n)\) is given by \eqref{eq:l_4(n)}. Then \(F_{j}^{(4)}(s,\rchi)\) admits a factorisation
\[
F_{j}^{(4)}(s, \rchi)=G_{j}^{(4)}(s, \rchi)\,H_{j}^{(4)}(s, \rchi),
\]
in which
\[
G_{j}^{(4)}(s,\rchi):=L\!\bigl(s,\operatorname{sym}^{j}f \otimes \rchi\bigr)\; 
        L\!\bigl(s-3,\operatorname{sym}^{j}f\otimes\rchi\bigr)
\]
and \(H_{j}^{(4)}(s, \rchi)\) is Dirichlet a series converges absolutely and uniformly 
in the half-plane \(\Re(s)>\frac{7}{2}\).
\end{lemma}

\begin{lemma}\label{lem:F5'}
Let \(f\) be a normalised primitive holomorphic cusp form of weight \(k\) for \(\operatorname{SL}(2,\mathbb{Z})\), 
and let \(\lambda_{\operatorname{sym}^{j}f}(n)\) denote the \(n\)-th normalised Fourier coefficient 
of the \(j\)-th symmetric power \(L\)-function attached to \(f\) and let $\rchi$ be a Dirichlet character modulo $q$.  

Define
\[
F_{j_1}^{(5)}(s, \rchi)=\sum_{n=1}^{\infty}\frac{\lambda_{\operatorname{sym}^{j}f}(n)\,l_5(n)\rchi(n)}{n^{s}} \qquad (\Re(s)>5),
\]
where \(l_5(n)\) is given by \eqref{eq:l_5(n)&v_5(n)}. Then \(F_{j_1}(s, \rchi)\) admits a factorisation
\[
F_{j_1}(s, \rchi)=G_{j_1}^{(5)}(s, \rchi)\,H_{j_1}^{(5)}(s, \rchi),
\]
in which 
\[
G_{j_1}^{(5)}(s,\rchi):=L\!\bigl(s-4,\operatorname{sym}^{j}f \otimes \rchi\bigr)\; 
        L\!\bigl(s,\operatorname{sym}^{j}f\otimes\rchi_4\rchi\bigr)
\]
and  \(H_{j_1}^{(5)}(s, \rchi)\) is a Dirichlet series that converges absolutely and uniformly 
in the half-plane \(\Re(s)>\frac{9}{2}\).
\end{lemma}

\begin{lemma}\label{lem:F5''}
        Let $f$ be a normalised primitive holomorphic cusp form of weight $k$ for $SL(2,\mathbb{Z})$, $\rchi$ be a Dirichlet character modulo $q$ and let $\lambda_{\mathrm{sym}^{j}f}(n)$ be the $n$-th normalised Fourier coefficient of the $j^{th}$ symmetric power $L$-function associated to $f$. Define
        \[
        F_{j_2}^{(5)}(s, \rchi)=\sum_{n=1}^{\infty}\frac{\lambda_{\mathrm{sym}^{j}f}(n)v_5(n)\rchi(n)}{n^{s}}, \quad \Re(s) > 5,
        \]
        where $v_5(n)$ is given by \eqref{eq:l_5(n)&v_5(n)}. Then $F_{j_2}^{(5)}(s,\rchi)$ admits a factorisation
        \[
        F_{j_2}^{(5)}(s,\rchi)=G_{j_2}^{(5)}(s,\rchi)H_{j_2}^{(5)}(s,\rchi),
        \]
        in which
        \[
        G_{j_2}^{(5)}(s,\rchi):=L(s, \mathrm{sym}_{j}f \otimes \rchi)L(s-4,\mathrm{sym}_jf \otimes\rchi_ 4\rchi)
        \]
        and $H_{j_2}^{(5)}(s,\rchi)$ is a Dirichlet series that converges uniformly and absolutely in the half plane $\Re(s)>\frac{9}{2}$.
    \end{lemma}

\begin{lemma}\label{lem:F6}
Let \(f\) be a normalised primitive holomorphic cusp form of weight \(k\) for \(\operatorname{SL}(2,\mathbb{Z})\), $\rchi$ be a Dirichlet character modulo $q$ and let \(\lambda_{\operatorname{sym}^{j}f}(n)\) denote the \(n\)-th normalised Fourier coefficient 
of the \(j\)-th symmetric power \(L\)-function attached to \(f\).  

Define
\[
F_{j}^{(6)}(s,\rchi)=\sum_{n=1}^{\infty}\frac{\lambda_{\operatorname{sym}^{j}f}(n)\,l_6(n)\rchi(n)}{n^{s}}, \qquad \Re(s)>6,
\]
where \(l_6(n)\) is a given by \eqref{eq:l_6(n)}. Then \(F_{j}^{(6)}(s, \rchi)\) admits a factorisation
\[
F_{j}^{(6)}(s,\rchi)=G_{j}^{(6)}(s,\rchi)\,H_{j}^{(6)}(s,\rchi),
\]
in which 
\[
G_{j}^{(6)}(s,\rchi):=L\!\bigl(s-5,\operatorname{sym}^{j}f \otimes \rchi\bigr)\; 
        L\!\bigl(s,\operatorname{sym}^{j}f \otimes \rchi\bigr)
\]
and \(H_{j}^{(6)}(s,\rchi)\) is a Dirichlet series that converges absolutely and uniformly 
in the half-plane \(\Re(s)>\frac{11}{2}\).
\end{lemma}

\begin{lemma} \label{lem:F1*}
Let \(f\) be a normalised primitive holomorphic cusp form of weight \(k\) for \(\operatorname{SL}(2,\mathbb{Z})\), 
 let \(\lambda_{\operatorname{sym}^{j}f}(n)\) denote the \(n\)-th normalised Fourier coefficient 
of the \(j\)-th symmetric power \(L\)-function attached to \(f\) and let $\rchi$ be a Dirichlet character modulo $q$.  

Define
\[
F_{j}^{(*1)}(s, \rchi)=\sum_{n=1}^{\infty}\frac{\lambda_{\operatorname{sym}^{j}f}^2(n)\,l_1(n)\rchi(n)}{n^{s}}, \qquad \Re(s)>1,
\]
where \(l_1(n)\) is given by \eqref{eq:l_1(n)}. Then \(F_{j}^{(*1)}(s,\rchi)\) admits a factorisation
\[
F_{j}^{(*1)}(s, \rchi)=G_{j}^{(*1)}(s, \rchi)\,H_{j}^{(*1)}(s, \rchi),
\]
in which
\[
G_{j}^{(*1)}(s,\rchi):=L(s,\rchi)L(s,\rchi_ 4\rchi)\prod_{n=1}^{j}L\!\bigl(s,\operatorname{sym}^{2n}f \otimes \rchi\bigr)\; 
        L\!\bigl(s,\operatorname{sym}^{2n}f\otimes\rchi_ 4\rchi\bigr)
\]
and \(H_{j}^{(*1)}(s, \rchi)\) is a Dirichlet series that converges absolutely and uniformly 
in the half-plane \(\Re(s)>\frac{1}{2}\).
\end{lemma}

\begin{proof}
    We observe that $\lambda_{\mathrm{sym}^{j}f}^{2}(n)l_1(n)\rchi(n)$ is multiplicative and hence
\begin{align} \label{eq:FjEuler}
F_{j}^{(*1)}(s)=\prod_{p}\left(1+\frac{\lambda_{\mathrm{sym}^{j}f}^{2}(p)l_1(p)\rchi(p)}{p^{s}}+\cdots+\frac{\lambda_{\mathrm{sym}^{j}f}^{2}(p^{m})l_1(p^{m})\rchi_ (p^m)}{p^{ms}}+\cdots\right).
\end{align}

Using \eqref{eq:hecke} and \eqref{lambdasquare}, we note that,
\begin{align*}
\lambda_{\mathrm{sym}^{j}f}^{2}(p)l_1(p)\rchi(p) &= \lambda_{f}^{2}(p^{j})\left(1+\rchi_ 4(p)\right)\rchi(p) \\
&= \left(1+\sum_{l=1}^{j}\lambda_{f}(p^{2l})\right)\left(1+\rchi_ 4(p)\right)\rchi(p) \\
&= \left(1+\sum_{l=1}^{j}\lambda_{\mathrm{sym}^{2l}f}(p)\right)\left(1+\rchi_ 4(p)\right)\rchi(p) \\
&= \rchi(p)+\rchi_ 4\rchi(p)+\sum_{l=1}^{j}\lambda_{\mathrm{sym}^{2l}f}(p)\rchi(p)+\sum_{l=1}^{j}\lambda_{\mathrm{sym}^{2l}f}(p)\rchi_ 4\rchi(p) \\
&=: b_1^*(p). \quad \text{(say)}
\end{align*}

From the structure of $b_1^*(p)$, we define the coefficients $b_1^*(n)$ as
\[
\sum_{n=1}^{\infty}\frac{b_1^*(n)}{n^{s}}= L(s,\rchi)L(s,\rchi_ 4\rchi)\prod_{n=1}^{j}L(s,\mathrm{sym}^{2n}f\otimes \rchi)L(s,\mathrm{sym}^{2n}f\otimes\rchi_ 4\rchi),
\]
which is absolutely convergent in $\Re(s)>1$. We also note that,
\begin{align*}
\prod_{p}\left(1+\frac{b^*(p)}{p^{s}}+\cdots+\frac{b^*(p^{m})}{p^{ms}}+\cdots\right) &= L(s,\rchi)L(s,\rchi_ 4\rchi)\prod_{n=1}^{j}L(s,\mathrm{sym}^{2n}f\otimes \rchi)L(s,\mathrm{sym}^{2n}f\otimes\rchi_ 4\rchi) \\
&= G_{j}^*(s,\rchi), \quad \text{(say)}
\end{align*}
for $\Re(s)>1$. Observe that $b_1^*(n)\ll_{\epsilon}n^{\epsilon}$ for any small positive constant $\epsilon$. The rest will be exactly as in Lemma~\ref{lem:F1}.
\end{proof}

\begin{lemma}\label{lem:F2*}
Let \(f\) be a normalised primitive holomorphic cusp form of weight \(k\) for \(\operatorname{SL}(2,\mathbb{Z})\), 
 let \(\lambda_{\operatorname{sym}^{j}f}(n)\) denote the \(n\)-th normalised Fourier coefficient 
of the \(j\)-th symmetric power \(L\)-function attached to \(f\) and let $\rchi$ be a Dirichlet character modulo $q$.  

Define
\[
F_{j}^{(*2)}(s, \rchi)=\sum_{n=1}^{\infty}\frac{\lambda_{\operatorname{sym}^{j}f}^2(n)\,l_2(n)\rchi(n)}{n^{s}}, \qquad \Re(s)>2,
\]
where \(l_2(n)\) is given by \eqref{eq:l_2(n)}. Then \(F_{j}^{(*2)}(s,\rchi)\) admits a factorisation
\[
F_{j}^{(*2)}(s, \rchi)=G_{j}^{(*2)}(s, \rchi)\,H_{j}^{(*2)}(s, \rchi),
\]
in which
\[
G_{j}^{(*2)}(s,\rchi):=L(s,\rchi)L(s-1,\rchi)\prod_{n=1}^{j}L\!\bigl(s,\operatorname{sym}^{2n}f \otimes \rchi\bigr)\; 
        L\!\bigl(s-1,\operatorname{sym}^{2n}f\otimes\rchi\bigr)
\]
and \(H_{j}^{(*2)}(s, \rchi)\) is a Dirichlet series that converges absolutely and uniformly 
in the half-plane \(\Re(s)>\frac{3}{2}\).
\end{lemma}

\begin{lemma}\label{lem:F3*'}
Let \(f\) be a normalised primitive holomorphic cusp form of weight \(k\) for \(\operatorname{SL}(2,\mathbb{Z})\), 
 let \(\lambda_{\operatorname{sym}^{j}f}(n)\) denote the \(n\)-th normalised Fourier coefficient 
of the \(j\)-th symmetric power \(L\)-function attached to \(f\) and let $\rchi$ be a Dirichlet character modulo $q$.  

Define
\[
F_{j_1}^{(*3)}(s, \rchi)=\sum_{n=1}^{\infty}\frac{\lambda_{\operatorname{sym}^{j}f}^2(n)\,l_3(n)\rchi(n)}{n^{s}}, \qquad \Re(s)>3,
\]
where \(l_3(n)\) is given by \eqref{eq:l_3(n)&v_3(n)}. Then \(F_{j_1}^{(*3)}(s,\rchi)\) admits a factorisation
\[
F_{j_1}^{(*3)}(s, \rchi)=G_{j_1}^{(*3)}(s, \rchi)\,H_{j_1}^{(*3)}(s, \rchi),
\]
in which
\[
G_{j_1}^{(*3)}(s,\rchi):=L(s,\rchi_ 4\rchi)L(s-2,\rchi)\prod_{n=1}^{j}L\!\bigl(s,\operatorname{sym}^{2n}f \otimes \rchi_ 4\rchi\bigr)\; 
        L\!\bigl(s-2,\operatorname{sym}^{2n}f\otimes\rchi\bigr)
\]
and \(H_{j_1}^{(*3)}(s, \rchi)\) is a Dirichlet series that converges absolutely and uniformly 
in the half-plane \(\Re(s)>\frac{5}{2}\).
\end{lemma}

\begin{lemma}\label{lem:F3*''}
Let \(f\) be a normalised primitive holomorphic cusp form of weight \(k\) for \(\operatorname{SL}(2,\mathbb{Z})\), 
 let \(\lambda_{\operatorname{sym}^{j}f}(n)\) denote the \(n\)-th normalised Fourier coefficient 
of the \(j\)-th symmetric power \(L\)-function attached to \(f\) and let $\rchi$ be a Dirichlet character modulo $q$.  

Define
\[
F_{j_2}^{(*3)}(s, \rchi)=\sum_{n=1}^{\infty}\frac{\lambda_{\operatorname{sym}^{j}f}^2(n)\,v_3(n)\rchi(n)}{n^{s}}, \qquad \Re(s)>3,
\]
where \(v_3(n)\) is given by \eqref{eq:l_3(n)&v_3(n)}. Then \(F_{j_2}^{(*3)}(s,\rchi)\) admits a factorisation
\[
F_{j_2}^{(*3)}(s, \rchi)=G_{j_2}^{(*3)}(s, \rchi)\,H_{j_2}^{(*3)}(s, \rchi),
\]
in which
\[
G_{j_2}^{(*3)}(s,\rchi):=L(s-2,\rchi_ 4\rchi)L(s,\rchi)\prod_{n=1}^{j}L\!\bigl(s-2,\operatorname{sym}^{2n}f \otimes \rchi_ 4\rchi\bigr)\; 
        L\!\bigl(s,\operatorname{sym}^{2n}f\otimes\rchi\bigr)
\]
and \(H_{j_2}^{(*3)}(s, \rchi)\) is a Dirichlet series that converges absolutely and uniformly 
in the half-plane \(\Re(s)>\frac{5}{2}\).
\end{lemma}

\begin{lemma} \label{lem:F4*}
Let \(f\) be a normalised primitive holomorphic cusp form of weight \(k\) for \(\operatorname{SL}(2,\mathbb{Z})\), 
and let \(\lambda_{\operatorname{sym}^{j}f}(n)\) denote the \(n\)-th normalised Fourier coefficient 
of the \(j\)-th symmetric power \(L\)-function attached to \(f\) and let $\rchi$ be a Dirichlet character modulo $q$.  

Define
\[
F_{j}^{(*4)}(s, \rchi)=\sum_{n=1}^{\infty}\frac{\lambda_{\operatorname{sym}^{j}f}^2(n)\,l_4(n)\rchi(n)}{n^{s}}, \qquad \Re(s)>4,
\]
where \(l_4(n)\) is given by \eqref{eq:l_4(n)}. Then \(F_{j}^{(*4)}(s,\rchi)\) admits a factorisation
\[
F_{j}^{(*4)}(s, \rchi)=G_{j}^{(*4)}(s, \rchi)\,H_{j}^{(*4)}(s, \rchi),
\]
in which
\[
G_{j}^{(*4)}(s,\rchi):=L(s,\rchi)L(s-3,\rchi)\prod_{n=1}^{j}L\!\bigl(s,\operatorname{sym}^{2n}f \otimes \rchi\bigr)\; 
        L\!\bigl(s-3,\operatorname{sym}^{2n}f\otimes\rchi\bigr)
\]
and \(H_{j}^{(*2)}(s, \rchi)\) is a Dirichlet series that converges absolutely and uniformly 
in the half-plane \(\Re(s)>\frac{7}{2}\).
\end{lemma}

\begin{lemma} \label{lem:F5*'}
Let $f$ be a normalised primitive holomorphic cusp form of weight $k$ for $SL(2,\mathbb{Z})$, $\rchi$ be a Dirichlet character modulo $q$ and let $\lambda_{\mathrm{sym}^{j}f}(n)$ be the $n$-th normalised Fourier coefficient of the $j^{th}$ symmetric power $L$-function associated to $f$. Define
\[
F_{j_1}^{(*5)}(s, \rchi)=\sum_{n=1}^{\infty}\frac{\lambda_{\mathrm{sym}^{j}f}^{2}(n)l_5(n) \rchi(n)}{n^{s}}, \quad \Re(s)>5,
\]
where $l_5(n)$ is given by \eqref{eq:l_5(n)&v_5(n)}. Then $F_{j_1}^{(*5)}(s,\rchi)$ admits a factorisation
\[
F_{j_1}^{(*5)}(s,\rchi)=G_{j_1}^{(*5)}(s,\rchi)H_{j_1}^{(*5)}(s,\rchi),
\]
in which
\[
G_{j_1}^{(*5)}(s,\rchi):=L(s-4, \rchi)L(s,\rchi_ 4)\prod_{n=1}^{j}L(s-4,\mathrm{sym}^{2n}f\otimes \rchi)L(s,\mathrm{sym}^{2n}f\otimes\rchi\rchi_ 4)
\]
and $H_{j_1}^{(*5)}(s)$ is a Dirichlet series that converges uniformly and absolutely in the half plane $\Re(s)>\frac{9}{2}$. 
\end{lemma}

\begin{lemma} \label{lem:F5*''}
Let $f$ be a normalised primitive holomorphic cusp form of weight $k$ for $SL(2,\mathbb{Z})$, $\rchi$ be a Dirichlet character modulo $q$, and let $\lambda_{\mathrm{sym}^{j}f}(n)$ be the $n$-th normalised Fourier coefficient of the $j^{th}$ symmetric power $L$-function associated to $f$. Define
\[
F_{j_2}^{(*5)}(s,\rchi)=\sum_{n=1}^{\infty}\frac{\lambda_{\mathrm{sym}^{j}f}^{2}(n)v(n)\rchi(n)}{n^{s}}, \quad \Re(s)>5,
\]
where $v_5(n)$ is given by \eqref{eq:l_5(n)&v_5(n)}. Then $F_{j_2}^{(*5)}(s,\rchi)$ admits a factorisation
\[
F_{j_2}^{(*5)}(s,\rchi)=G_{j_2}^{(*5)}(s,\rchi)H_{j_2}^{(*5)}(s,\rchi),
\]
in which
\[
G_{j_2}^{(*5)}(s,\rchi):=L(s,\rchi)L(s-4,\rchi_ 4\rchi)\prod_{n=1}^{j}L(s,\mathrm{sym}^{2n}f\otimes \rchi)L(s-4,\mathrm{sym}^{2n}f\otimes\rchi_ 4\rchi)
\]
and $H_{j_2}^{(*5)}(s,\rchi)$ is a Dirichlet series that converges uniformly and absolutely in the half plane $\Re(s)>\frac{9}{2}$.
\end{lemma}

\begin{lemma} \label{lem:F6*}
Let $f$ be a normalised primitive holomorphic cusp form of weight $k$ for $SL(2,\mathbb{Z})$, $\rchi$ be a Dirichlet character modulo $q$ and let $\lambda_{\mathrm{sym}^{j}f}(n)$ be the $n$-th normalised Fourier coefficient of the $j^{th}$ symmetric power $L$-function associated to $f$. Define
\[
F_{j}^{(*6)}(s,\rchi)=\sum_{n=1}^{\infty}\frac{\lambda_{\mathrm{sym}^{j}f}^{2}(n)l_6(n)\rchi(n)}{n^{s}}, \quad \Re(s)>6,
\]
where $l_6(n)$ is given by \eqref{eq:l_6(n)}. Then $F_{j}^{(*6)}(s, \rchi)$ admits a factorisation
\[
F_{j}^{(*6)}(s, \rchi)=G_{j}^{(*6)}(s,\rchi)H_{j}^{(*6)}(s,\rchi),
\]
in which
\[
G_{j}^{(*6)}(s,\rchi):=L(s-5,\rchi)L(s,\rchi)\prod_{n=1}^{j}L(s-5,\mathrm{sym}^{2n}f \otimes \rchi)L(s,\mathrm{sym}^{2n}f \otimes \rchi)
\]
and $H_{j}^{(*6)}(s, \rchi)$ is a Dirichlet series that converges uniformly and absolutely in the half plane $\Re(s)>\frac{11}{2}$. 
\end{lemma}

\begin{lemma}\cite{Venkat} \label{lem:PrincipalChar}
    Let $\rchi_ 0$ be the principal character modulo $q$. Then we have $$L(s,\rchi_ 0) = \zeta(s)\prod_{p|q}\left( 1- \frac{1}{p^s} \right)$$ and $$L(s, \mathrm{sym}^{2n}f\otimes \rchi_ 0) = L(s, \mathrm{sym}^{2n}f)\prod_{p|q}\prod_{0\leq j \leq 2n}\left( 1- \frac{\alpha_f^{2n-2j}(p)}{p^s} \right)$$ for all integers $n \geq 1$ and 
    $$\prod_{p|q}\left( 1- \frac{1}{p^s} \right) \ll q^{\epsilon} ,$$
    $$\prod_{\substack{p|q }}\prod_{0\leq j \leq 2n}\left( 1- \frac{\alpha_f^{2n-2j}(p)}{p^s} \right) \ll q^{\epsilon},$$
    for $\frac{1}{2}+\epsilon < \Re(s) < 1+\epsilon$.
\end{lemma}

\begin{lemma} \cite{Venkat} \label{lem:PrimitiveChar}
    Let $\rchi$ be a non-primitive character modulo q and $\rchi^*$ be a primitive character modulo $q_1 (\neq q)$ induced by $\rchi$. Then we have $$L(s,\rchi) = L(s,\rchi^*)\prod_{\substack{p|q \\ p \nmid q_1}}\left( 1- \frac{\rchi^*(p)}{p^s} \right),$$ $$L(s, \mathrm{sym}^{2n}f\otimes \rchi) = L(s, \mathrm{sym}^{2n}f \otimes \rchi^*)\prod_{\substack{p|q \\ p \nmid q_1}}\prod_{0\leq j \leq 2n}\left( 1- \frac{\alpha_f^{2n-2j}(p)\rchi^*(p)}{p^s} \right),$$ for all integers $n \geq 1$ and 
    $$\prod_{\substack{p|q \\ p \nmid q_1}}\left( 1- \frac{\rchi^*(p)}{p^s} \right) \ll q^\epsilon,$$
    $$\prod_{\substack{p|q \\ p \nmid q_1}}\prod_{0\leq j \leq 2n}\left( 1- \frac{\alpha_f^{2n-2j}(p)\rchi^*(p)}{p^s} \right) \ll q^{\epsilon},$$
    for $\frac{1}{2}+\epsilon < \Re(s) < 1+\epsilon$.
\end{lemma}

\begin{lemma}\label{lem:LSigmaChi}
    Let $\rchi$ be any primitive character modulo $q$. Then for $q \ll T^2$, we have \begin{align}
        L(\sigma+iT,\rchi) \ll (q(1+|T|))^{\mathrm{max}\{\frac{1}{3}(1-\sigma),0\}+\epsilon}
    \end{align}
    holds uniformly for $\frac{1}{2} \leq \sigma \leq 2$ ; and \begin{align}
        \int_{1}^{T}|L(\sigma+iT,\rchi)|^4 \ll (qT)^{2(1-\sigma)+\epsilon}
    \end{align}
    uniformly for $\frac{1}{2}\leq \sigma \leq 1+\epsilon$ and $T \geq 1$.
\end{lemma}
\begin{proof}
    The first result follows from \cite{Heath} and the second follows from \cite{Per}.
\end{proof}

\begin{lemma} \label{lem:zetabound}
For any $\frac{1}{2}\leq \sigma \leq 2$ and $T \geq 2$, we have
\begin{align} \label{eq:zeta12mom}
\int_{1}^{T}\left|\zeta\left(\sigma+it\right)\right|^{4}dt\ll T^{1+\epsilon},
\end{align}
and \begin{align} \label{eq:zetabound}
\zeta(\sigma+it)\ll_{\epsilon}(1+|t|)^{\max\{\frac{13}{42}(1-\sigma),0\}+\epsilon},
\end{align}
uniformly for $\frac{1}{2}\leq\sigma\leq 1+\epsilon$, and $|t|\geq 1$.
\end{lemma}
\begin{proof}
For the proof of \eqref{eq:zeta12mom}, see \cite{Titch}, page-$148$ and \eqref{eq:zetabound} is due to Bourgain, for instance, see \cite{BourJe}.
\end{proof}

\begin{lemma} \label{lem:LSymFChi}
    Let $f \in S_{\kappa}$, and $\rchi$ be a primitive character modulo $q$. Then for $q \ll T^2$, we have 
    \begin{align} \label{eq:LSigmafbound}
        L(\sigma+iT,\mathrm{sym}^2f) \ll (1+|T|)^{\mathrm{max} \{\frac{6}{5}(1-\sigma),0\} + \epsilon}
    \end{align}
    and \begin{align}\label{eq:LSigmaSymChibound}
        L(\sigma+iT,\mathrm{sym}^2f \otimes\rchi) \ll (q(1+|T|))^{\mathrm{max} \{\frac{67}{46}(1-\sigma),0\} + \epsilon}
    \end{align}
    uniformly for $\frac{1}{2} \leq \sigma \leq 2$ and $|T| \geq 1$;
    \begin{align}\label{eq:intLSigmafChi}
        \int_{1}^{T}|L(\sigma+iT,\mathrm{sym}^2f \otimes\rchi)|^4dt \ll (qT)^{6(1-\sigma)+\epsilon}
    \end{align}
    uniformly for $\frac{1}{2} \leq \sigma \leq 1+\epsilon$ and $T \geq 1$.
 \end{lemma}

 \begin{proof}
     \eqref{eq:LSigmafbound} and \eqref{eq:LSigmaSymChibound} follow from Phragm\'{en}-Lindel\"{o}f convexity principle and the work of Lin, Nunes
     and Qi \cite{Lin} and Huang \cite{Huang} respectively. \eqref{eq:intLSigmafChi} follows from Perelli \cite{Per}.
 \end{proof}

\begin{lemma} \label{lem:genLbound}
Let $\rchi$ be a primitive character modulo $q$ and $\mathfrak{L}_{m,n}^{d}(s,\rchi)$ be a general $L$-function of degree $2A$. For any $\epsilon>0$, we have
\begin{align} \label{eq:genL2mom}
\int_{T}^{2T}\left|\mathfrak{L}_{m,n}^{d}(\sigma+it,\rchi)\right|^{2}dt\ll (qT)^{2A(1-\sigma)+\epsilon},
\end{align}
uniformly for $\frac{1}{2}\leq\sigma\leq 1+\epsilon$, and $T\geq 1$. Also,
\begin{align} \label{eq:genLbound}
\mathfrak{L}_{m,n}^{d}(\sigma+it,\rchi)\ll (q(1+|t|))^{\max\{A(1-\sigma),0\}+\epsilon},
\end{align}
uniformly for $-\epsilon\leq\sigma\leq 1+\epsilon$.
\end{lemma}
\begin{proof}
For the proof of \eqref{eq:genL2mom} and \eqref{eq:genLbound}, see \cite{JiangLu}.
\end{proof}
\end{subsection}

\begin{lemma} \label{lem:int}
    Let $f:\mathbb{R} \mapsto \mathbb{R}$, and $T > 1$. Then \begin{align}
        \int_{1}^{T}\frac{|f(t)|}{t}dt \ll \log T \sup_{1 \leq T_1 \leq T} \frac{1}{T_1}\int_{T_1}^{2T_1}|f(t)|dt.
    \end{align}
\end{lemma}

\begin{lemma}\label{dominating main term}
Let \(f: \mathbb{R} \to \mathbb{R}_{>0}\) satisfy
\[
f(x) = Dx^{A} + O(x^{B})
\]
as \(x \to \infty\), where \(A, B \in \mathbb{R}\), \( D>0\), and \(A > B\).  
Then there exists \(X_0 > 0\) such that for all \(x > X_0\),
\[
f(x) \ge \frac{D}{2} x^{A}.
\]

\end{lemma}

\begin{lemma} \label{lem:r_2chi}
Let $j\geq 2$ be a fixed integer. For any fixed $\epsilon>0$ and all sufficiently large $x$, we have
    \begin{align}
        \sum_{\substack{n\leq x }}\lambda_{\mathrm{sym}^jf}(n)\rchi(n)r_2(n) = \begin{cases}
           O \left( x^{1-\frac{1}{j+1}+\epsilon}q^{\frac{1}{2}+\epsilon}\right) \quad \text{ when } \rchi = \rchi_0 \text{ or } \rchi \rchi_4 = \rchi_0, \\
           O \left( x^{1-\frac{1}{j+1}+\epsilon}q^{1+\epsilon}\right) \quad \text{ when } \rchi \neq \rchi_0 \text{ and } \rchi \rchi_4 \neq \rchi_0.
        \end{cases}
    \end{align}
\end{lemma}

\begin{proof}
Let $\rchi$ be a Dirichlet character modulo $q$. We estimate the term $\sum_{\substack{n\leq x }}\lambda_{\mathrm{sym}^jf}(n)\rchi_0(n)l_1(n)$. As in the previous paper, We begin by applying Perron's formula to $F_{j}^{(1)}(s)$ with $\eta=1+\epsilon$, and $10\leq T\leq x$. Thus, we have,
\begin{align}
    \sum_{\substack{n\leq x }}\lambda_{\mathrm{sym}^jf}(n)\rchi_0(n)l_1(n) & = \int_{1+\epsilon-iT}^{1+\epsilon+iT}F_{j}^{(1)}(s,\rchi_ 0)\frac{x^s}{s}ds + O\left( \frac{x^{1+\epsilon}}{T}\right).
\end{align}

After moving the line of integration to $\Re(s)=\frac{1}{2}+\epsilon$, by Cauchy's residue theorem, there are no poles due to the Lemma~\ref{lem:F1}.
So we obtain,

\begin{align} \label{eq:LambdaEq}
    \sum_{\substack{n\leq x}}\lambda_{\mathrm{sym}^{j}f}(n)l_1(n)\rchi_ 0(n)&= \frac{1}{2\pi i}\left\{\int_{\frac{1}{2}+\epsilon-iT}^{\frac{1}{2}+\epsilon+iT}+\int_{1+\epsilon-iT}^{\frac{1}{2}+\epsilon-iT}+\int_{\frac{1}{2}+\epsilon+iT}^{1+\epsilon+iT}\right\}F_{j}^{(1)}(s,\rchi_ 0)\frac{x^{s}}{s}ds \\
    &\quad + O\left(\frac{x^{1+\epsilon}}{T}\right) \\
    &= \frac{1}{2\pi i}(J_{1}+J_{2}+J_{3})+O\left(\frac{x^{1+\epsilon}}{T}\right). \quad \text{(say)} 
\end{align}

Contribution of horizontal line integrals ($J_{2}$ and $J_{3}$) in absolute value (using Lemmas~\ref{lem:F1} and~\ref{lem:genLbound}) is

\begin{align}
    |J_2+J_3| &= \left| \left( \int_{1+\epsilon-iT}^{\frac{1}{2}+\epsilon-iT}+\int_{\frac{1}{2}+\epsilon+iT}^{1+\epsilon+iT} \right)F_j^{(1)}(s, \rchi_ 0)\frac{x^s}{s}\right| \\
    &\ll  \left( \int_{\frac{1}{2}+\epsilon}^{1+\epsilon} + \int_{\frac{1}{2}+\epsilon}^{1+\epsilon}\right) \frac{|L(\sigma+iT, \mathrm{sym}^j f \otimes \rchi_ 0) L(\sigma+iT, \mathrm{sym}^j f \otimes \rchi_ 4\rchi_ 0)|}{T}x^\sigma d\sigma,
\end{align}

\begin{align*}
    J_2+J_3 &\ll \int_{\frac{1}{2}+\epsilon}^{1+\epsilon} \frac{|T|^{\frac{j+1}{2}(1-\sigma)+\epsilon}|qT|^{\frac{j+1}{2}(1-\sigma)+\epsilon} q^\epsilon}{T}x^{\sigma}d\sigma \quad \text{ (using~\ref{lem:genLbound}) } \\
&\ll \frac{1}{T} \max_{\frac{1}{2}+\epsilon < \sigma < 1+\epsilon} \left( x^\sigma q^{\frac{j+1}{2}(1-\sigma)+\epsilon}T^{(j+1)(1-\sigma)+\epsilon}\right) .
\end{align*}

Clearly, $x^\sigma q^{\frac{j+1}{2}(1-\sigma)+\epsilon}T^{(j+1)(1-\sigma)+\epsilon}$ is a monotonic function, so the maximum occurs at the endpoints of the interval. We take values at both extreme points of the interval $[\frac{1}{2}+\epsilon, 1+\epsilon]$. So 
\begin{align*}
    J_2+J_3 &\ll \frac{1}{T}\left( x^{\frac{1}{2}+\epsilon}q^{\frac{j+1}{4}+\epsilon}T^{\frac{(j+1)}{2} +\epsilon)} \right) +  \frac{1}{T}\left( x^{1+\epsilon}(qT)^{\epsilon}
    \right) \\
    &\ll \frac{x^{1+\epsilon}}{T}q^{\epsilon}+x^{\frac{1}{2}+\epsilon}q^{\frac{j+1}{4}+\epsilon}T^{\frac{j+1}{2}-1+\epsilon}.
\end{align*}

Now contribution of vertical line integral $J_1$ in absolute value is \begin{align*}
    J_1 &= \int_{\frac{1}{2}+\epsilon - iT}^{\frac{1}{2}+\epsilon+iT}F_j^{(1)}(s, \rchi_ 0) \frac{x^{\frac{1}{2}+\epsilon+it}}{\frac{1}{2}+\epsilon+it}ds \\
    &= x^{\frac{1}{2}+\epsilon}\left( \int_{0\leq |t|\leq 1} + \int_{1 \leq |t| \leq T} \right)F_j^{(1)}\left(\frac{1}{2}+\epsilon+it, \rchi_ 0\right) \frac{x^{it}}{\frac{1}{2}+\epsilon+it}idt\\
    &= I_1+I_2.
\end{align*} 

Now using the Lemma~\ref{lem:genLbound}, we have
\begin{align*}
    I_2 &\ll x^{\frac{1}{2}+\epsilon} \int_{1}^{T}|L(\frac{1}{2}+\epsilon+it, \mathrm{sym}^j f \otimes \rchi_ 0) L(\frac{1}{2}+\epsilon+it, \mathrm{sym}^j f \otimes \rchi_ 4\rchi_ 0)| \frac{1}{t} dt \\
    &\ll x^{\frac{1}{2}+\epsilon} \log T\text{ }\sup_{1 \leq T_1 \leq T}\frac{1}{T_1}\int_{T_1}^{2T_1}|L(\frac{1}{2}+\epsilon+it, \mathrm{sym}^j f \otimes \rchi_ 0) L(\frac{1}{2}+\epsilon+it, \mathrm{sym}^j f \otimes \rchi_ 4\rchi_ 0)| dt \\
    & \ll x^{\frac{1}{2}+\epsilon}\log T\text{ }\sup_{1 \leq T_1 \leq T}\frac{1}{T_1} \left( \int_{T_1}^{2T_1}|L(\frac{1}{2}+\epsilon+it, \mathrm{sym}^j f\otimes \rchi_ 0)|^2 dt \right)^{\frac{1}{2}} \\
    & \times\left( \int_{T_1}^{2T_1} |L(\frac{1}{2}+\epsilon+it, \mathrm{sym}^j f\otimes \rchi_ 4\rchi_ 0)|^2 dt \right)^{\frac{1}{2}} \\
    & \ll x^{\frac{1}{2}+\epsilon} \sup_{1 \leq T_1 \leq T}\frac{1}{T_1}\left(T_1^{\max\{(j+1)(1-\frac{1}{2}-\epsilon), 0\} + \epsilon})^{\frac{1}{2}}\right)\left((qT_1)^{\max\{(j+1)(1-\frac{1}{2}-\epsilon), 0\} + \epsilon})^{\frac{1}{2}}\right)q^{\epsilon} \quad \text{(using~\ref{lem:genLbound})} \\
    & \ll x^{\frac{1}{2}+\epsilon}q^{\frac{j+1}{4}+\epsilon}T^{\frac{j+1}{2}-1+\epsilon}.
\end{align*}

The first integral gives \begin{align*}
    I_1 &= x^{\frac{1}{2}+\epsilon}\int_{0 \leq |t| \leq 1}F_j^{(1)}\left(\frac{1}{2}+\epsilon+it,\rchi_0\right) \frac{x^{it}}{\frac{1}{2}+\epsilon+it}dt.
    \end{align*}
    The above integration is finite. If not, then the left-hand side of \eqref{eq:LambdaEq} would be infinite. As the other integral is finite, this is a contradiction. So,
    \begin{align*}
        I_1 \ll x^{\frac{1}{2}+\epsilon}.
    \end{align*}
    
    Combining $I_1$ and $I_2$, we have \begin{align}
    J_1 \ll x^{\frac{1}{2}+\epsilon} + x^{\frac{1}{2}+\epsilon} q^{\frac{j+1}{4}}T^{\frac{j+1}{2}-1+\epsilon}.
    \end{align}
    
    Thus we have \begin{align}
    J_1+J_2+J_3 \ll \frac{x^{1+\epsilon}}{T}q^{\epsilon}+x^{\frac{1}{2}+\epsilon} q^{\frac{j+1}{4}+\epsilon}T^{\frac{j+1}{2}-1+\epsilon}.
    \end{align}

    So $\sum_{\substack{n\leq x }}\lambda_{\mathrm{sym}^jf}(n)\rchi_0(n)l_1(n) = O \left(\frac{x^{1+\epsilon}}{T}q^{\epsilon}+x^{\frac{1}{2}+\epsilon} q^{\frac{j+1}{4}+\epsilon}T^{\frac{j+1}{2}-1+\epsilon} \right)$.

     Now put $T = \frac{x^{\frac{1}{j+1}}}{q^{\frac{1}{2}}}$, then \begin{align}
        \sum_{\substack{n\leq x}}\lambda_{\mathrm{sym}^jf}(n)\rchi_0(n)l_1(n) = O \left( x^{1-\frac{1}{j+1}+\epsilon}q^{\frac{1}{2}+\epsilon}\right).
    \end{align}
    Note that the same bound of  $\sum_{\substack{n\leq x}}\lambda_{\mathrm{sym}^jf}(n)\rchi(n)l_1(n)$ happens when $\rchi \rchi_4 = \rchi_0$,
    and after doing the same calculation for $\sum_{\substack{n\leq x}}\lambda_{\mathrm{sym}^jf}(n)\rchi(n)l_1(n)$, where $\rchi$ is any primitive or non primitive Dirichlet charachter modulo $q$ and $\rchi \rchi_4 \neq \rchi_0$, we have \begin{align}
        \sum_{\substack{n\leq x}}\lambda_{\mathrm{sym}^jf}(n)\rchi(n)l_1(n) = O \left( x^{1-\frac{1}{j+1}+\epsilon}q^{1+\epsilon}\right).
    \end{align}
    
\end{proof}

\begin{lemma} \label{lem:r4681012chi}
Let $j\geq 2$ be a fixed integer. For any fixed $\epsilon>0$ and all sufficiently large $x$, we have
    \begin{align}
        \sum_{\substack{n\leq x}}\lambda_{\mathrm{sym}^jf}(n)\rchi(n)r_m(n) = \begin{cases}
            O \left( x^{\frac{m}{2}-\frac{2}{j+1}+\epsilon}q^{\epsilon} \right) \quad \text{ when } \rchi = \rchi_0, \\
            O \left( x^{\frac{m}{2}-\frac{2}{j+3}+\epsilon}q^{\frac{j+1}{j+3}+\epsilon} \right) \quad \text{ when } \rchi \neq \rchi_0.
        \end{cases}
    \end{align}
\end{lemma}

\begin{proof}
We estimate the sum $\sum_{\substack{n\leq x}}\lambda_{\mathrm{sym}^jf}(n)\rchi_0(n)l_2(n)$.

Now By Perron's formula, we have \begin{align}
 \sum_{\substack{n\leq x}}\lambda_{\mathrm{sym}^jf}(n)\rchi_0(n)l_2(n) & = \int_{2+\epsilon-iT}^{2+\epsilon+iT}F_{j}^{(2)}(s,\rchi_ 0)\frac{x^s}{s}ds + O\left( \frac{x^{2+\epsilon}}{T}\right).
\end{align}

We move the line of integration to $\Re(s) = \frac{3}{2}+\epsilon$ and by the Cauchy residue theorem, we get that there exists no pole in the area of integration due to the Lemma~\ref{lem:F2}.

\begin{align}
    \sum_{\substack{n\leq x}}\lambda_{\mathrm{sym}^jf}(n)\rchi_0(n)l_2(n) & = \frac{1}{2\pi i}\left\{\int_{\frac{3}{2}+\epsilon-iT}^{\frac{3}{2}+\epsilon+iT}+\int_{2+\epsilon-iT}^{\frac{3}{2}+\epsilon-iT}+\int_{\frac{3}{2}+\epsilon+iT}^{2+\epsilon+iT}\right\}F_{j}^{(2)}(s,\rchi_ 0)\frac{x^{s}}{s}ds \\
    &\quad + O\left(\frac{x^{2+\epsilon}}{T}\right) \\
    &= \frac{1}{2\pi i}(J_{1}+J_{2}+J_{3})+O\left(\frac{x^{2+\epsilon}}{T}\right). \quad \text{(say)} 
\end{align}

Contribution of horizontal line integrals ($J_{2}$ and $J_{3}$) in absolute value (using Lemmas~\ref{lem:F2},~\ref{lem:PrincipalChar} and~\ref{lem:genLbound}) is
\begin{align*}
|J_2+J_3| &= \left| \left( \int_{2+\epsilon-iT}^{\frac{3}{2}+\epsilon-iT}+\int_{\frac{3}{2}+\epsilon+iT}^{2+\epsilon+iT} \right)F_j^{(2)}(s,\rchi_ 0)\frac{x^s}{s}\right| \\
&\ll  \left( \int_{\frac{3}{2}+\epsilon}^{2+\epsilon} + \int_{\frac{3}{2}+\epsilon}^{2+\epsilon}\right) \frac{|L(\sigma+iT-1, \mathrm{sym}^j f\otimes \rchi_ 0)|}{T}x^\sigma d\sigma .
\end{align*}

The above inequality happens because $F_j^{(2)}(s,\rchi_ 0)=G_j^{(2)}(s,\rchi_ 0)H_j^{(2)}(s,\rchi_ 0)$ and $H_j^{(2)}(s, \rchi_ 0) \ll 1$ for $\Re(s) > \frac{3}{2}.$ So $F_j^{(2)}(s,\rchi_ 0) \ll G_j^{(2)}(s,\rchi_ 0) = L(s-1, \mathrm{sym}^j f \otimes \rchi_ 0)L(s, \mathrm{sym}^j f \otimes \rchi_ 0)$. Now $L(s, \mathrm{sym}^j f \otimes \rchi_ 0)$ is absolutely convergent for $\Re(s) > 1$. So $L(s, \mathrm{sym}^j f \otimes \rchi_ 0) \ll 1$ in $\Re(s) > \frac{3}{2}$.

\begin{align*}
    J_2+J_3 &\ll \int_{\frac{3}{2}+\epsilon}^{2+\epsilon}\frac{|L(\sigma-1+iT , \mathrm{sym}^j f)|}{T} 
x^\sigma q^{\epsilon} d\sigma \quad (\text{ by Lemma~\ref{lem:PrincipalChar}})\\
&\ll \int_{\frac{1}{2}+\epsilon}^{1+\epsilon} \frac{|L(\sigma+iT , \mathrm{sym}^j f)|}{T}x^{\sigma+1}q^{\epsilon}d\sigma \\
&\ll \int_{\frac{1}{2}+\epsilon}^{1+\epsilon} \frac{|T|^{\frac{j+1}{2}(1-\sigma)+\epsilon}}{T}x^{\sigma+1}q^{\epsilon}d\sigma \quad \text{ (using~\ref{lem:genLbound})} \\
&\ll \frac{x}{T}q^{\epsilon} \max_{\frac{1}{2}+\epsilon < \sigma < 1+\epsilon} \left( x^\sigma T^{\frac{j+1}{2}(1-\sigma)+\epsilon}\right) .
\end{align*}
Clearly, $x^\sigma T^{\frac{j+1}{2}(1-\sigma)+\epsilon}$ is a monotonic function, so the maximum occurs at the endpoints of the interval. We take values at both extreme points of the interval $[\frac{1}{2}+\epsilon, 1+\epsilon]$. So 
\begin{align*}
    J_2+J_3 &\ll \frac{x}{T}q^{\epsilon}\left( x^{\frac{1}{2}+\epsilon}T^{\frac{j+1}{2} (1-\frac{1}{2}-\epsilon) +\epsilon} \right) +  \frac{x}{T}q^{\epsilon}\left( x^{1+\epsilon}T^{\frac{j+1}{2} (1-1-\epsilon) +\epsilon}
    \right) \\
    &\ll \frac{x^{2+\epsilon}}{T}q^{\epsilon}+x^{\frac{3}{2}+\epsilon}T^{\frac{j+1}{4}-1+\epsilon}q^{\epsilon}.
\end{align*}

Now contribution of vertical line integral $J_1$ in absolute value is \begin{align*}
    J_1 &= \int_{\frac{3}{2}+\epsilon - iT}^{\frac{3}{2}+\epsilon+iT}F_j^{(2)}(s,\rchi_ 0) \frac{x^{\frac{3}{2}+\epsilon+it}}{\frac{3}{2}+\epsilon+it}ds \\
    &= (x+1)^{\frac{3}{2}+\epsilon}\left( \int_{0\leq |t|\leq 1} + \int_{1 \leq |t| \leq T} \right)F_j^{(2)}\left(\frac{3}{2}+\epsilon+it,\rchi_ 0\right) \frac{x^{it}}{\frac{3}{2}+\epsilon+it}idt\\
    &= I_1+I_2.
\end{align*}

Now \begin{align*}
    I_2 &\ll x^{\frac{3}{2}+\epsilon}q^{\epsilon} \int_{1}^{T}|L(\frac{1}{2}+\epsilon+it, \mathrm{sym}^j f)| \frac{1}{t} dt \quad ( \text{ by Lemma }~\ref{lem:PrincipalChar} \text{ and }~\ref{lem:F2})\\
    &\ll x^{\frac{3}{2}+\epsilon}q^{\epsilon} \log T\text{ }\sup_{1 \leq T_1 \leq T}\frac{1}{T_1}\int_{T_1}^{2T_1}|L(\frac{1}{2}+\epsilon+it, \mathrm{sym}^j f)| dt ( \text{ by Lemma }~\ref{lem:int}  \\
    & \ll x^{\frac{3}{2}+\epsilon}q^{\epsilon}\log T\text{ }\sup_{1 \leq T_1 \leq T}\frac{1}{T_1} \left( \int_{T_1}^{2T_1}|L(\frac{1}{2}+\epsilon+it, \mathrm{sym}^j f)|^2 dt \right)^{\frac{1}{2}}\left( \int_{T_1}^{2T_1} 1 dt \right)^{\frac{1}{2}} \\
    & \ll x^{\frac{3}{2}+\epsilon} q^{\epsilon}\sup_{1 \leq T_1 \leq T}\frac{1}{T_1}(T_1^{\max\{(j+1)(1-\frac{1}{2}-\epsilon), 0\} + \epsilon})^{\frac{1}{2}}T_1^\frac{1}{2} \quad \text{(using~\ref{lem:genLbound})} \\
    & \ll x^{\frac{3}{2}+\epsilon}T^{\frac{j+1}{4}-\frac{1}{2}+\epsilon}q^{\epsilon}.
\end{align*}

The first integral gives \begin{align*}
    I_1 &= x^{\frac{3}{2}+\epsilon}\int_{0 \leq |t| \leq 1}F_j^{(2)}\left(\frac{3}{2}+\epsilon+it, \rchi_ 0\right) \frac{x^{it}}{\frac{3}{2}+\epsilon+it}dt.
    \end{align*}
    The above integration is finite. If not, then $\sum_1$ would be infinite. As the other integral is finite, this is a contradiction. So,
    \begin{align*}
        I_1 \ll x^{\frac{3}{2}+\epsilon}.
    \end{align*}

Combining $I_1$ and $I_2$, we have \begin{align}
    J_1 \ll x^{\frac{3}{2}+\epsilon} + x^{\frac{3}{2}+\epsilon} T^{\frac{j+1}{4}-\frac{1}{2}+\epsilon}q^{\epsilon}.
\end{align}
Thus we have \begin{align}
    J_1+J_2+J_3 \ll \frac{x^{2+\epsilon}}{T}q^{\epsilon}+x^{\frac{3}{2}+\epsilon} T^{\frac{j+1}{4}-\frac{1}{2}+\epsilon}q^{\epsilon}.
\end{align}

We have \begin{align} 
      \sum_{n\leq x}\lambda_{\mathrm{sym}^{j}f}(n)l_2(n)\rchi_ 0(n) = O\left(\frac{x^{2+\epsilon}}{T}q^{\epsilon}+x^{\frac{3}{2}+\epsilon} T^{\frac{j+1}{4}-\frac{1}{2}+\epsilon}q^{\epsilon} \right).
\end{align}

Put $T = x^{\frac{2}{j+3}}$, then \begin{align}
    \sum_{n\leq x}\lambda_{\mathrm{sym}^{j}f}(n)l_2(n)\rchi_ 0(n) = O \left( x^{2-\frac{2}{j+3}+\epsilon}q^{\epsilon} \right).
\end{align}

Now for $\sum_{n\leq x}\lambda_{\mathrm{sym}^{j}f}(n)l_2(n)\rchi(n)$, where $\rchi$ is a primitive Dirichlet character modulo $q$, using the lemmas~\ref{lem:PrimitiveChar},~\ref{lem:genLbound}, we have \begin{align}
    J_2+J_3 &\ll \int_{\frac{3}{2}+\epsilon}^{2+\epsilon}\frac{|L(\sigma-1+iT , \mathrm{sym}^j f \otimes \rchi)|}{T}
x^\sigma d\sigma  \\
&\ll \frac{x^{2+\epsilon}}{T}q^{\epsilon}+x^{\frac{3}{2}+\epsilon}q^{\frac{j+1}{4}+\epsilon}T^{\frac{j+1}{4}-1+\epsilon}q^{\epsilon}
\end{align}
and \begin{align}
    J_1 &\ll \left( \int_{0}^{1}+\int_{1}^{T} \right)\left|L\left(\frac{1}{2}+\epsilon+it, \mathrm{sym}^jf \otimes \rchi\right)\right|\frac{dt}{t} \\
    &\ll x^{\frac{3}{2}+\epsilon} + x^{\frac{3}{2}+\epsilon} q^{\frac{j+1}{4}}T^{\frac{j+1}{4}-\frac{1}{2}+\epsilon}.
\end{align}
Thus \begin{align}
    J_1+J_2+J_3 \ll \frac{x^{2+\epsilon}}{T}q^{\epsilon} + x^{\frac{3}{2}+\epsilon} q^{\frac{j+1}{4}+\epsilon}T^{\frac{j+1}{4}-\frac{1}{2}+\epsilon}.
\end{align}

Thus \begin{align}
    \sum_{n\leq x}\lambda_{\mathrm{sym}^{j}f}(n)l_2(n)\rchi(n) = O\left( \frac{x^{2+\epsilon}}{T}q^{\epsilon}+x^{\frac{3}{2}+\epsilon}q^{\frac{j+1}{4}+\epsilon}T^{\frac{j+1}{4}-\frac{1}{2}+\epsilon} \right).
\end{align}

Put $T = \frac{x^{\frac{2}{j+3}}}{q^{\frac{j+1}{j+3}}}$. Then \begin{align}
    \sum_{n\leq x}\lambda_{\mathrm{sym}^{j}f}(n)l_2(n)\rchi(n) = O \left( x^{2-\frac{2}{j+3}+\epsilon}q^{\frac{j+1}{j+3}+\epsilon} \right),
\end{align}
for any non primitive Dirichlet character modulo $q$.

Now for $m=6$, We have, \begin{align}
    \sum_{n\leq x}\lambda_{\mathrm{sym}^{j}f}(n)l_3(n)\rchi_0(n) = O\left( \frac{x^{3+\epsilon}}{T}q^{\epsilon}+x^{\frac{5}{2}+\epsilon}q^{\epsilon}T^{\frac{j+1}{4}-\frac{1}{2}+\epsilon} \right)
\end{align}
and 
\begin{align}
    \sum_{n\leq x}\lambda_{\mathrm{sym}^{j}f}(n)v_3(n)\rchi_0(n) = O\left( \frac{x^{3+\epsilon}}{T}q^{\epsilon}+x^{\frac{5}{2}+\epsilon}q^{\epsilon}T^{\frac{j+1}{4}-\frac{1}{2}+\epsilon} \right).
\end{align}
From \eqref{eq:r_6&l_3&v_3}, we have \begin{align}
    \sum_{n\leq x}\lambda_{\mathrm{sym}^{j}f}(n)r_6(n)\rchi_0(n) = O\left( \frac{x^{3+\epsilon}}{T}q^{\epsilon}+x^{\frac{5}{2}+\epsilon}q^{\epsilon}T^{\frac{j+1}{4}-\frac{1}{2}+\epsilon} \right).
\end{align}

Put $T = x^{\frac{2}{j+3}}$, then \begin{align}
    \sum_{n\leq x}\lambda_{\mathrm{sym}^{j}f}(n)r_6(n)\rchi_ 0(n) = O \left( x^{3-\frac{2}{j+3}+\epsilon}q^{\epsilon} \right).
\end{align}

For $\chi \neq \chi_0$, \begin{align}
    \sum_{n\leq x}\lambda_{\mathrm{sym}^{j}f}(n)v_3(n)\rchi(n) = O\left( \frac{x^{3+\epsilon}}{T}q^{\epsilon}+x^{\frac{5}{2}+\epsilon}q^{\frac{j+1}{4}+\epsilon}T^{\frac{j+1}{4}-\frac{1}{2}+\epsilon} \right),
\end{align}
and
\begin{align}
    \sum_{n\leq x}\lambda_{\mathrm{sym}^{j}f}(n)v_3(n)\rchi(n) = O\left( \frac{x^{3+\epsilon}}{T}q^{\epsilon}+x^{\frac{5}{2}+\epsilon}q^{\frac{j+1}{4}+\epsilon}T^{\frac{j+1}{4}-\frac{1}{2}+\epsilon} \right)
\end{align}

and from \eqref{eq:r_6&l_3&v_3}, we have \begin{align}
    \sum_{n\leq x}\lambda_{\mathrm{sym}^{j}f}(n)r_6(n)\rchi(n) = O\left( \frac{x^{3+\epsilon}}{T}q^{\epsilon}+x^{\frac{5}{2}+\epsilon}q^{\frac{j+1}{4}+\epsilon}T^{\frac{j+1}{4}-\frac{1}{2}+\epsilon} \right).
\end{align}

Put $T = \frac{x^{\frac{2}{j+1}}}{q^{\frac{j+1}{j+3}}}$. Then \begin{align}
    \sum_{n\leq x}\lambda_{\mathrm{sym}^{j}f}(n)r_6(n)\rchi(n) = O \left( x^{3-\frac{2}{j+3}+\epsilon}q^{\frac{j+1}{j+3}+\epsilon} \right),
\end{align}
for any non primitive Dirichlet character modulo $q$.

Following the above process, we will have the same bounds for $\sum_{n\leq x}\lambda_{\mathrm{sym}^{j}f}(n)l_{\frac{m}{2}}(n)\rchi(n)$, for $m=6,8,10,12$ and $\sum_{n\leq x}\lambda_{\mathrm{sym}^{j}f}(n)v_{\frac{m}{2}}(n)\rchi(n)$ for $m=6,10$.  

For the sum of $10$ squares, we have an additional term $\lambda_{\mathrm{sym}^jf}(n)\rchi(n)a_n$ beside the terms $\lambda_{\mathrm{sym}^jf}(n)\rchi(n)l_5(n)$ and $\lambda_{\mathrm{sym}^jf}(n)\rchi(n)v_5(n)$, as we can see from \eqref{eq:Lmbda10Def}. We have from \cite{Grosswald} that $a_n = O(n^3)$. As shown in a parallel paper, we need only to compute the terms involving $l_5(n)$ and $v_5(n)$. 

For the sum of $12$ squares, we have an additional term $\lambda_{\mathrm{sym}^jf}(n)\rchi(n)b_n$ beside the term $\lambda_{\mathrm{sym}^jf}(n)\rchi(n)l_6(n)$,as we can see from \eqref{eq:Lmbda12Def} and we have from \cite{Grosswald} that $b_n=O(n^3\log \log n)$. As shown in a parallel paper, we only need to calculate the term with $l_6(n)$.

Thus we have, using \eqref{eq:r_4&l_2}, \eqref{eq:r_8&l_4} and \eqref{eq:r_12&l_6}, we have
\begin{align}
    \sum_{n\leq x}\lambda_{\mathrm{sym}^{j}f}(n)r_m(n)\rchi_ 0(n) = O \left( x^{\frac{m}{2}-\frac{2}{j+3}+\epsilon}q^{\epsilon} \right),
\end{align}
and
\begin{align}
    \sum_{n\leq x}\lambda_{\mathrm{sym}^{j}f}(n)r_m(n)\rchi(n) &= O \left( x^{\frac{m}{2}-\frac{2}{j+3}+\epsilon}q^{\frac{j+1}{j+3}+\epsilon} \right).
\end{align}
\end{proof}

\begin{lemma} \label{lem:r_2chi2}
Let $j\geq 2$ be a fixed integer. For any fixed $\epsilon>0$ and all sufficiently large $x$, we have
    \begin{align*}
        \sum_{\substack{n\leq x}}\lambda_{\mathrm{sym}^jf}^2(n)\rchi(n)r_2(n) &= \begin{cases}
            A \frac{\phi(q)}{q}x + O \left( x^{1-\frac{1}{(j+1)^2}+\epsilon}q^{\frac{1}{2}+\epsilon} \right), \quad \text{ when } \rchi = \rchi_0 \text{ or } \rchi \rchi_4 = \rchi_0, \\
            O \left( x^{1-\frac{1}{(j+1)^2}+\epsilon}q^{1+\epsilon} \right), \quad \text{ when } \rchi \neq \rchi \text{ and } \rchi \rchi_4 \neq \rchi_0,
        \end{cases} 
    \end{align*}
    where $A$ is a constant.
\end{lemma}

\begin{proof}
We consider the sum $$\sum_{\substack{n\leq x}}\lambda_{\mathrm{sym}^jf}^2(n)l_1(n).$$ We begin by applying Perron's formula to $\sum_{\substack{n\leq x}}\lambda_{\mathrm{sym}^jf}^2(n)l_1(n)$ with $\eta=1+\epsilon$, and $10\leq T\leq x$. Thus, we have

\begin{align*}
\sum_{n\leq x} \lambda_{\mathrm{sym}^jf}^2(n)l_1(n)\rchi_ 0(n)
= \frac{1}{2\pi i}\int_{\eta-iT}^{\eta+iT}F_{j}^{(*1)}(s,\rchi_ 0)\frac{x^{s}}{s}ds + O\left(\frac{x^{1+\epsilon}}{T}\right).
\end{align*}

We move the line of integration to $\Re(s)=\frac{1}{2}+\epsilon$, and by Cauchy's residue theorem, there is only one simple pole at $s=1$ due to the factor $L(s, \rchi_ 0)$, we get from $F^{(*1)}_{j}(s)$ in the Lemma~\ref{lem:F1*}. 
This contributes a residue, which is $\frac{\phi(q)}{q}c_{j,f}(1)(x+1)$, where 
\begin{align} \label{eq:cjf}
c_{j,f}(1) &= L(1,\rchi_ 4\rchi_ 0)\prod_{n=1}^{j}L(1,\mathrm{sym}^{2n}f\otimes \rchi_ 0)L(1,\mathrm{sym}^{2n}f\otimes\rchi_ 4\rchi_ 0)H_{j}^{(*1)}(1,\rchi_ 0).
\end{align}

So, we obtain
\begin{align*}
\sum_{n\leq x} \lambda_{\mathrm{sym}^jf}^2(n)l_1(n)\rchi_ 0(n) &= \frac{\phi(q)}{q}c_{j,f}(1)x
\\ &+ \frac{1}{2\pi i}\left\{\int_{\frac{1}{2}+\epsilon-iT}^{\frac{1}{2}+\epsilon+iT}+\int_{1+\epsilon-iT}^{\frac{1}{2}+\epsilon-iT}+\int_{\frac{1}{2}+\epsilon+iT}^{1+\epsilon+iT}\right\}F_{j}^{(*1)}(s,\rchi_ 0)\frac{x^{s}}{s}ds + O\left(\frac{x^{1+\epsilon}}{T}\right) \\
&= \frac{\phi(q)}{q}c_{j,f}(1)x + \frac{1}{2\pi i}(J_{1}+J_{2}+J_{3})+O\left(\frac{x^{1+\epsilon}}{T}\right). \quad \text{(say)}
\end{align*}

Contribution of horizontal line integrals ($J_{2}$ and $J_{3}$) in absolute value (using Lemmas~\ref{lem:F1*},~\ref{lem:PrincipalChar},~\ref{lem:PrimitiveChar},\ref{lem:zetabound},~\ref{lem:LSymFChi},~\ref{lem:LSigmaChi} and~\ref{lem:genLbound}) is

\begin{align*}
    J_2+J_3 &\ll \left| \left( \int_{1+\epsilon-iT}^{\frac{1}{2}+\epsilon-iT}+\int_{\frac{1}{2}+\epsilon+iT}^{1+\epsilon+iT} \right)F_j^{(*1)}(s,\rchi_ 0)\frac{x^s}{s} ds \right| \\
    &\ll  \int_{\frac{1}{2}+\epsilon}^{1+\epsilon} \left|\zeta(\sigma+iT)L(\sigma+iT, \rchi_ 4\rchi_ 0^*)\prod_{n=1}^{j}L(\sigma+iT, \mathrm{sym}^{2n} f)L(\sigma+iT, \mathrm{sym}^{2n} f \otimes \rchi_ 4 \rchi_ 0^*)\right|\frac{1}{T}x^\sigma q^{\epsilon} d\sigma \\
    &\ll \frac{1}{T}\max_{\frac{1}{2}+\epsilon \leq \sigma \leq 1+\epsilon} x^\sigma T^{(\frac{13}{42}+\frac{6}{5}+\sum_{2\leq n\leq j}\frac{2n+1}{2})(1-\sigma)+\epsilon}(qT)^{(\frac{1}{3}+\frac{67}{46}+\sum_{2\leq n\leq j}\frac{2n+1}{2})(1-\sigma)+\epsilon}\\
    &\ll \frac{1}{T}\max_{\frac{1}{2}+\epsilon \leq \sigma \leq 1+\epsilon}x^\sigma q^{\left(\frac{(j+1)^2}{2}-\frac{29}{138}\right)(1-\sigma)+\epsilon} T^{\left((j+1)^2-\frac{564}{805} \right )(1-\sigma)+\epsilon}.
\end{align*}

The above function involving $\sigma$ is monotonic, so the maximum happens at the endpoints. We treat both boundary points as upper bounds.
\begin{align*}
    J_2+J_3 &\ll \frac{x^{1+\epsilon}}{T}q^{\epsilon} + x^{\frac{1}{2}+\epsilon}q^{\frac{(j+1)^2}{4}-\frac{29}{276}+\epsilon}T^{\frac{(j+1)^2}{2}-\frac{1087}{805}+\epsilon} .
\end{align*}

Contribution of the left vertical line integral ($J_{1}$) in absolute value (using Lemmas~\ref{lem:F1*},~\ref{lem:PrincipalChar},~\ref{lem:PrimitiveChar},~\ref{lem:LSigmaChi},~\ref{eq:intLSigmafChi}~\ref{lem:zetabound},~\ref{lem:genLbound} and H\"older's inequality) is

\begin{align*}
J_1 &\ll \int_{\frac{1}{2}+\epsilon-iT}^{\frac{1}{2}+\epsilon+iT}\left|\zeta(\tfrac{1}{2}+\epsilon+it)L(\tfrac{1}{2}+\epsilon+it, \rchi_ 4\rchi_ 0^*)\right| \\
& \times \left|\prod_{n=1}^{j}L(\tfrac{1}{2}+\epsilon+it,\mathrm{sym}^{2n}f )L(\tfrac{1}{2}+\epsilon+it,\mathrm{sym}^{2n}f \otimes \rchi_ 4\rchi_ 0^* )\right|\frac{x^{\frac{1}{2}+\epsilon}}{|t|}q^{\epsilon}dt \\
&\ll x^{\frac{1}{2}+\epsilon} + x^{\frac{1}{2}+\epsilon}\int_{1 \leq |t| \leq T}\left|\zeta(\tfrac{1}{2}+\epsilon+it)L(\tfrac{1}{2}+\epsilon+it, \rchi_ 4\rchi_ 0^*) \right| \\ 
& \times \left|\prod_{n=1}^{j}L(\tfrac{1}{2}+\epsilon+it,\mathrm{sym}^{2n}f)L(\tfrac{1}{2}+\epsilon+it,\mathrm{sym}^{2n}f \otimes \rchi_ 4\rchi_ 0^* )\right|\frac{1}{|t|}q^{\epsilon}dt \\
&\ll  x^{\frac{1}{2}+\epsilon} + x^{\frac{1}{2}+\epsilon} q^{\epsilon}\log T \text{ } \sup_{1 \leq T_1 \leq T}\frac{1}{T_1} \int_{T_1}^{2T_1}\left|\zeta(\tfrac{1}{2}+\epsilon+it)L(\tfrac{1}{2}+\epsilon+it, \rchi_ 4\rchi_ 0^*) \right| \\ 
& \times \left|\prod_{n=1}^{j}L(\tfrac{1}{2}+\epsilon+it,\mathrm{sym}^{2n}f)L(\tfrac{1}{2}+\epsilon+it,\mathrm{sym}^{2n}f \otimes \rchi_ 4\rchi_ 0^* )\right| dt \\
&= x^{\frac{1}{2}+\epsilon} + x^{\frac{1}{2}+\epsilon}I_2,
\end{align*}

where the bounds of $I_2$ is given by as follows,

\begin{align}
    I_2 &=q^{\epsilon}\log T \text{ } \sup_{1 \leq T_1 \leq T}\frac{1}{T_1} \int_{T_1}^{2T_1}\left|\zeta(\tfrac{1}{2}+\epsilon+it)L(\tfrac{1}{2}+\epsilon+it, \rchi_ 4\rchi_ 0^*) \right| \\ 
    & \times \left|\prod_{n=1}^{j}L(\tfrac{1}{2}+\epsilon+it,\mathrm{sym}^{2n}f)L(\tfrac{1}{2}+\epsilon+it,\mathrm{sym}^{2n}f \otimes \rchi_ 4\rchi_ 0^* )\right| dt \\
    &\ll q^{\epsilon}T^\epsilon \sup_{1 \leq T_1 \leq T}\frac{1}{T_1}\left(  \int_{T_1}^{2T_1}|\zeta(\frac{1}{2}+\epsilon+it)|^{4}dt \right)^\frac{1}{4}\left(\int_{T_1}^{2T_1}L(\tfrac{1}{2}+\epsilon+it, \rchi_ 4\rchi_ 0^*)^4dt\right)^\frac{1}{4} \\
    & \times \left( \int_{T_1}^{2T_1}\left| \prod_{n=2}^{j}L(\tfrac{1}{2}+\epsilon+it,\mathrm{sym}^{2n}f)L(\tfrac{1}{2}+\epsilon+it,\mathrm{sym}^{2n}f \otimes \rchi_ 4\rchi_ 0^* ) \right|^{2}dt\right)^{\frac{1}{2}}\\
    &\ll q^{\epsilon}\frac{1}{T_1}T_1^{(\frac{1}{4}+\epsilon) + 2(\frac{1}{2}-\epsilon)\frac{1}{4}+ 2 \sum_{1 \leq n \leq j}(2n+1)(\frac{1}{2}-\epsilon)\frac{1}{2}}q^{2(\frac{1}{2}-\epsilon)\frac{1}{4}+\sum_{1 \leq n \leq j}(2n+1)(\frac{1}{2}-\epsilon)\frac{1}{2}}\\
    &= q^{\frac{(j+1)^2}{4}+\epsilon} T^{\frac{(j+1)^2}{2}-1+\epsilon}.
\end{align}

So we have \begin{align}
    J_1 \ll x^{\frac{1}{2}+\epsilon} + x^{\frac{1}{2}+\epsilon}q^{\frac{(j+1)^2}{4}+\epsilon} T^{\frac{(j+1)^2}{2}-1+\epsilon},
\end{align}

and \begin{align}
    J_1+J_2+J_3 \ll \frac{x^{1+\epsilon}}{T}q^{\epsilon}+ x^{\frac{1}{2}+\epsilon}q^{\frac{(j+1)^2}{4}+\epsilon} T^{\frac{(j+1)^2}{2}-1+\epsilon}.
\end{align}

Thus \begin{align}
    \sum_{n\leq x} \lambda_{\mathrm{sym}^jf}^2(n)l_1(n)\rchi_ 0(n) = \frac{\phi(q)}{q}c_{j,f}(1)x + O \left( \frac{x^{1+\epsilon}}{T}q^{\epsilon}+ x^{\frac{1}{2}+\epsilon}q^{\frac{(j+1)^2}{4}+\epsilon} T^{\frac{(j+1)^2}{2}-1+\epsilon} \right).
\end{align}

After we put $T = \frac{x^{\frac{1}{(j+1)^2}}}{q^{\frac{1}{2}}}$, we have \begin{align}
    \sum_{n\leq x} \lambda_{\mathrm{sym}^jf}^2(n)l_1(n)\rchi_ 0(n) = \frac{\phi(q)}{q}c_{j,f}(1)x + O \left( x^{1-\frac{1}{(j+1)^2}+\epsilon}q^{\frac{1}{2}+\epsilon} \right).
\end{align}

For $\sum_{n\leq x} \lambda_{\mathrm{sym}^jf}^2(n)l_1(n)\rchi(n)$, where $\rchi$ is a non principal and primitive Dirichlet modulo $q$ and $\rchi \rchi_4$ is the principal character modulo $q$, then we have have \begin{align}
    \sum_{n\leq x} \lambda_{\mathrm{sym}^jf}^2(n)l_1(n)\rchi(n) = \frac{\phi(4q)}{4q}d_{j,f}(1)x + O \left( x^{1-\frac{1}{(j+1)^2}+\epsilon}q^{\frac{1}{2}+\epsilon} \right),
\end{align}
where \begin{align} \label{eq:djf}
d_{j,f}(1) &= L(1,\rchi)\prod_{n=1}^{j}L(1,\mathrm{sym}^{2n}f\otimes \rchi)L(1,\mathrm{sym}^{2n}f\otimes\rchi_ 4\rchi)H_{j}^{(*1)}(1,\rchi).
\end{align}

Proceeding as above, we have \begin{align}
    \sum_{n\leq x} \lambda_{\mathrm{sym}^jf}^2(n)l_1(n)\rchi(n) = \frac{\phi(4q)}{4q}d_{j,f}(1)x + O \left( x^{1-\frac{1}{(j+1)^2}+\epsilon}q^{\frac{1}{2}+\epsilon} \right),
\end{align}
where $\rchi \rchi_4 = \rchi_0$.

Now for $\sum_{n\leq x} \lambda_{\mathrm{sym}^jf}^2(n)l_1(n)\rchi(n)$, where $\rchi$ is a non principal but primitive Dirichlet modulo $q$, and $\rchi \rchi_4 \neq \rchi_0$, then we have  
\begin{align*}
\sum_{\substack{n\leq x}} \lambda_{\mathrm{sym}^jf}^2(n)l_1(n)\rchi(n)
&= \frac{1}{2\pi i}\int_{\eta-iT}^{\eta+iT}F_{j}^{(*1)}(s,\rchi)\frac{x^{s}}{s}ds + O\left(\frac{x^{1+\epsilon}}{T}\right) \\
& =\frac{1}{2\pi i}\left\{\int_{\frac{1}{2}+\epsilon-iT}^{\frac{1}{2}+\epsilon+iT}+\int_{1+\epsilon-iT}^{\frac{1}{2}+\epsilon-iT}+\int_{\frac{1}{2}+\epsilon+iT}^{1+\epsilon+iT}\right\}F_{j}^{(*1)}(s,\rchi)\frac{x^{s}}{s}ds \\
&\quad + O\left(\frac{x^{1+\epsilon}}{T}\right) \\
&= \frac{1}{2\pi i}(J_{1}+J_{2}+J_{3})+O\left(\frac{x^{1+\epsilon}}{T}\right).
\end{align*}

After doing a similar calculation, we have the same bounds for $J_1+J_2+J_3$ as in the previous case.

After doing a similar calculation, we have \begin{align}
    \sum_{\substack{n\leq x}} \lambda_{\mathrm{sym}^jf}^2(n)l_1(n)\rchi(n) = O \left( \frac{x^{1+\epsilon}}{T}q^{\epsilon}+x^{\frac{1}{2}+\epsilon}q^{\frac{(j+1)^2}{2}+\epsilon}T^{\frac{(j+1)^2}{2}-1+\epsilon} \right).
\end{align}

After we put $T = \frac{x^{\frac{1}{(j+1)^2}}}{q^{1+\epsilon}}$, we have \begin{align}
    \sum_{n\leq x} \lambda_{\mathrm{sym}^jf}^2(n)l_1(n)\rchi(n) =  O \left( x^{1-\frac{1}{(j+1)^2}+\epsilon}q^{1+\epsilon} \right),
\end{align}
where $\rchi \rchi_4 \neq \rchi_0$.

\end{proof}

\begin{lemma} \label{lem:r4812chi2}
Let $j\geq 2$ be a fixed integer. For $m=4,8,12$ and any fixed $\epsilon>0$ and all sufficiently large $x$, we have
    \begin{align}
    \lambda_{\mathrm{sym}^jf}^2(n)r_m(n)\rchi(n) = \begin{cases}
        A\frac{\phi(q)}{q}x^\frac{m}{2} + O \left( x^{\frac{m}{2}-\frac{2}{(j+1)^2}+\epsilon}q^{\epsilon} \right), \quad \text{ when } \rchi = \rchi_0, \\
        O\left( x^{\frac{m}{2}-\frac{2}{(j+1)^2}+\epsilon}q^{1+\epsilon} \right), \quad \text{ when } \rchi \neq \rchi_0.
    \end{cases} 
\end{align}
\end{lemma}

\begin{proof}
We first calculate $\sum_{n\leq x} \lambda_{\mathrm{sym}^jf}^2(n)l_2(n)\rchi_0(n)$. 

Using the Lemma~\ref{lem:F2*}, Cauchy residue theorem, 

\begin{align}
     \lambda_{\mathrm{sym}^jf}^2(n)l_2(n)\rchi_0(n) & = \frac{\phi(q)}{q}a_{j,f}(2)x^2+\frac{1}{2\pi i}\left\{\int_{\frac{3}{2}+\epsilon-iT}^{\frac{3}{2}+\epsilon+iT}+\int_{2+\epsilon-iT}^{\frac{3}{2}+\epsilon-iT}+\int_{\frac{3}{2}+\epsilon+iT}^{2+\epsilon+iT}\right\}F_{j}^{(*2)}(s,\rchi_ 0)\frac{x^{s}}{s}ds \\
    &\quad + O\left(\frac{x^{2+\epsilon}}{T}\right) \\
    &= \frac{\phi(q)}{q}a_{j,f}(2)x^2+\frac{1}{2\pi i}(J_{1}+J_{2}+J_{3})+O\left(\frac{x^{2+\epsilon}}{T}\right), \quad \text{(say)} 
\end{align}

where \begin{align} \label{ajf}
    a_{j,f}(2) &= \frac{1}{2}L(2,\rchi_ 0)\prod_{n=1}^{j}L(2,\mathrm{sym}^{2n}f \otimes \rchi_ 0)L(1,\mathrm{sym}^{2n}f \times \rchi_ 0)H_j^2(2,\rchi_ 0)
\end{align}

Now using the lemmas~\ref{lem:F2*},~\ref{lem:PrincipalChar},~\ref{lem:zetabound},~\ref{lem:LSymFChi},~\ref{lem:genLbound}, we have 
\begin{align}
    J_2+J_3 &\ll \left| \left( \int_{2+\epsilon-iT}^{\frac{3}{2}+\epsilon-iT}+\int_{\frac{3}{2}+\epsilon+iT}^{2+\epsilon+iT} \right)F_j^{*2}(s, \rchi_ 0)\frac{x^s}{s}\right| \\
    &\ll  \int_{\frac{3}{2}+\epsilon}^{2+\epsilon} \left|\zeta(\sigma-1+iT)\prod_{n=1}^{j}L(\sigma+iT-1, \mathrm{sym}^{2n} f)\right|\frac{1}{T}x^\sigma q^{\epsilon} d\sigma \\
    &= \int_{\frac{1}{2}+\epsilon}^{1+\epsilon} \left|\zeta(\sigma+iT)\prod_{n=1}^{j}L(\sigma +iT, \mathrm{sym}^{2n} f)\right|\frac{x^{\sigma+1}}{T}q^{\epsilon}d\sigma \\
    &\ll \frac{x}{T}q^{\epsilon}\max_{\frac{1}{2}+\epsilon \leq \sigma \leq 1+\epsilon} x^\sigma T^{(\frac{13}{42}+\frac{6}{5}+\sum_{2\leq n\leq j}\frac{2n+1}{2})(1-\sigma)+\epsilon}\\
    &\ll \frac{x}{T}q^{\epsilon}\max_{\frac{1}{2}+\epsilon \leq \sigma \leq 1+\epsilon}x^\sigma T^{(\frac{(j+1)^2}{2}-\frac{103}{210})(1-\sigma)+\epsilon}.
\end{align}

The above function involving $\sigma$ is monotonic, so the maximum happens at the extreme points. We take both boundary points as an upper bound.

\begin{align*}
    J_2+J_3 &\ll \frac{x}{T}q^{\epsilon}\left[ x^{(1+\epsilon)} T^{(\frac{(j+1)^2}{2}-\frac{103}{210})(\epsilon)+\epsilon} +x^{(\frac{1}{2}+\epsilon)} T^{(\frac{(j+1)^2}{2}-\frac{103}{210})(\frac{1}{2}-\epsilon)+\epsilon}\right] \\
    &\ll \frac{x^{2+\epsilon}}{T}q^{\epsilon} + x^{\frac{3}{2}+\epsilon}T^{\frac{(j+1)^2}{4}-\frac{523}{420}+\epsilon}q^{\epsilon} .
\end{align*}

Contribution of the left vertical line integral ($J_{1}$) in absolute value (using Lemmas~\ref{lem:F2*},~\ref{lem:LSigmaChi},~\ref{lem:zetabound},~\ref{lem:LSymFChi},~\ref{lem:genLbound} and H\"older's inequality) is

\begin{align*}
J_1 &\ll \int_{\frac{3}{2}+\epsilon-iT}^{\frac{3}{2}+\epsilon+iT}\left|\zeta(\tfrac{1}{2}+\epsilon+it)\prod_{n=1}^{j}L(\tfrac{1}{2}+\epsilon+it,\mathrm{sym}^{2n}f)\right|\frac{x^{\frac{3}{2}+\epsilon}}{|t|}dt \\
&\ll x^{\frac{3}{2}+\epsilon} + x^{\frac{3}{2}+\epsilon}\int_{1 \leq |t| \leq T}\left|\zeta(\tfrac{1}{2}+\epsilon+it)\prod_{n=1}^{j}L(\tfrac{1}{2}+\epsilon+it,\mathrm{sym}^{2n}f)\right|\frac{1}{|t|}dt \\
&\ll  x^{\frac{3}{2}+\epsilon} + x^{\frac{3}{2}+\epsilon} \log T \text{ } \sup_{1 \leq T_1 \leq T}\frac{1}{T_1} \int_{T_1}^{2T_1}\left|\zeta(\tfrac{1}{2}+\epsilon+it)\prod_{n=1}^{j}L(\tfrac{1}{2}+\epsilon+it,\mathrm{sym}^{2n}f)\right| dt \\
&= x^{\frac{3}{2}+\epsilon} + x^{\frac{3}{2}+\epsilon}I_2,
\end{align*}

where the bounds of $I_2$ is given by as follows,

\begin{align}
    I_2 &=q^{\epsilon}\log T \sup_{1 \leq T_1 \leq T} \frac{1}{T_1} \int_{T_1}^{2T_1}\left|\zeta(\tfrac{1}{2}+\epsilon+it)\prod_{n=1}^{j}L(\tfrac{1}{2}+\epsilon+it,\mathrm{sym}^{2n}f)\right| dt \\
    &\ll q^{\epsilon}T^\epsilon \sup_{1 \leq T_1 \leq T}\frac{1}{T_1}\left(  \int_{T_1}^{2T_1}|\zeta(\frac{1}{2}+\epsilon+it)|^{4}dt \right)^\frac{1}{4}\left(\int_{T_1}^{2T_1}|L(\frac{1}{2}+\epsilon+it, \mathrm{sym}^{2}f)|^4dt\right)^\frac{1}{4} \\
    & \times \left( \int_{T_1}^{2T_1}\left| \prod_{n=2}^{j}L(\frac{1}{2}+\epsilon+it, \mathrm{sym}^{2n} f) \right|^{2}dt\right)^{\frac{1}{2}}\\
    &\ll q^{\epsilon}T^{\frac{1}{4}+\epsilon + 6(\frac{1}{2}+\epsilon)\frac{1}{4}+ ((j+1)^2-4)(\frac{1}{2}-\epsilon)\frac{1}{2}-1} = q^{\epsilon}T^{\frac{(j+1)^2}{4}-1+\epsilon}.
\end{align}

Thus, we have \begin{align}
    J_1+J_2+J_3 \ll \frac{x^{2+\epsilon}}{T}q^{\epsilon} + x^{\frac{3}{2}+\epsilon}T^{\frac{(j+1)^2}{4}-1+\epsilon}q^{\epsilon}.
\end{align}

Thus \begin{align}
    \lambda_{\mathrm{sym}^jf}^2(n)l_2(n)\rchi_0(n) = \frac{\phi(q)}{q}a_{j,f}(2)x^2 + O \left( \frac{x^{2+\epsilon}}{T}q^{\epsilon}+x^{\frac{3}{2}+\epsilon} T^{\frac{(j+1)^2}{4}-1+\epsilon}q^{\epsilon} \right).
\end{align}

Now put $T= x^{\frac{2}{(j+1)^2}}$, then we have \begin{align}
    \lambda_{\mathrm{sym}^jf}^2(n)l_2(n)\rchi_0(n) = \frac{\phi(q)}{q}a_{j,f}(2)x^2 + O \left( x^{2-\frac{2}{(j+1)^2}+\epsilon}q^{\epsilon} \right).
\end{align}

Now for $\lambda_{\mathrm{sym}^jf}^2(n)l_2(n)\rchi(n)$, where $\rchi$ is a non principal but a primitive Dirichlet character modulo $q$, we have by using the Lemma~\ref{lem:F2*} and by Cauchy residue theorem and by Perron's formula, we have \begin{align}
    \lambda_{\mathrm{sym}^jf}^2(n)l_2(n)\rchi(n) &=\frac{1}{2\pi i}\left\{\int_{\frac{3}{2}+\epsilon-iT}^{\frac{3}{2}+\epsilon+iT}+\int_{2+\epsilon-iT}^{\frac{3}{2}+\epsilon-iT}+\int_{\frac{3}{2}+\epsilon+iT}^{2+\epsilon+iT}\right\}F_{j}^{(*2)}(s,\rchi)\frac{x^{s}}{s}ds \\
    &\quad + O\left(\frac{x^{2+\epsilon}}{T}\right) \\
    &= \frac{1}{2\pi i}(J_{1}+J_{2}+J_{3})+O\left(\frac{x^{2+\epsilon}}{T}\right). \quad \text{(say)} .
\end{align}

Proceeding as before, we have by using the Lemmas~\ref{lem:F2*},~\ref{lem:PrimitiveChar}, ~\ref{lem:LSigmaChi},~\ref{lem:LSymFChi} and~\ref{lem:genLbound}), we have

 \begin{align}
    J_1+J_2+J_3 \ll  \frac{x^{2+\epsilon}}{T}q^{\epsilon} + x^{\frac{3}{2}+\epsilon}q^{^{\frac{(j+1)^2}{4}+\epsilon}}T^{\frac{(j+1)^2}{4}-1+\epsilon}.
\end{align}

If we put $T = \frac{x^{\frac{2}{(j+1)^2}}}{q}$, then we have \begin{align}
    \lambda_{\mathrm{sym}^jf}^2(n)l_2(n)\rchi(n) &= O\left( x^{2-\frac{2}{(j+1)^2}+\epsilon}q^{1+\epsilon} \right).
\end{align}

Similarly we have for $m=8,12$, \begin{align}
    \lambda_{\mathrm{sym}^jf}^2(n)l_{\frac{m}{2}}(n)\rchi_0(n) = \frac{\phi(q)}{q}a_{j,f}(\frac{m}{2})x^\frac{m}{2} + O \left( x^{\frac{m}{2}-\frac{2}{(j+1)^2}+\epsilon}q^{\epsilon} \right)
\end{align}
and 
\begin{align}
    \lambda_{\mathrm{sym}^jf}^2(n)l_{\frac{m}{2}}(n)\rchi(n) &= O\left( x^{\frac{m}{2}-\frac{2}{(j+1)^2}+\epsilon}q^{1+\epsilon} \right),
\end{align}
where $\rchi$ is not a principal Dirichlet character modulo $q$.

\end{proof}

\begin{lemma} \label{lem:r610chi2}
Let $j\geq 2$ be a fixed integer. For $m=6,10$ and for any fixed $\epsilon>0$ and all sufficiently large $x$, we have
    \begin{align} 
    \lambda_{\mathrm{sym}^jf}^2(n)r_m(n)(n)\rchi(n) = \begin{cases}
        A\frac{\phi(q)}{q}x^\frac{m}{2} +  O\left(x^{\frac{m}{2}-\frac{2}{(j+1)^2}+\epsilon}q^{1+\epsilon} \right) , \quad \text{ when } \rchi = \rchi_0 \text{ or } \rchi\rchi_4 = \rchi_0, \\
        O\left(x^{\frac{m}{2}-\frac{2}{(j+1)^2}+\epsilon}q^{1+\epsilon} \right), \quad \text{ when } \rchi \neq \rchi_0 \text{ and } \rchi\rchi_4 \neq \rchi_0.
    \end{cases}
\end{align}
\end{lemma}
\begin{proof}
We estimate $\lambda_{\mathrm{sym}^jf}^2(n)v_3(n)\rchi(n)$, where $\rchi \rchi_4 = \rchi_0$. Then, using the above procedure, we have   \begin{align}
     \lambda_{\mathrm{sym}^jf}^2(n)v_3(n)\rchi(n) = \frac{\phi(4q)}{4q}b''_{j,f}(3)x^3 + O\left(\frac{x^{2+\epsilon}}{T}q^{\epsilon}+x^{\frac{3}{2}+\epsilon} T^{\frac{(j+1)^2}{4}-1+\epsilon}q^{\epsilon} \right),
\end{align}

where \begin{align} \label{eq:b''jf}
    b''_{j,f}(3) &= \frac{1}{3}L(3,\rchi)\prod_{n=1}^{j}L(3,\mathrm{sym}^{2n}f\otimes \rchi)L(1, \mathrm{sym}^{2n}f \otimes \rchi_ 4\rchi)H_j^{(3)}(3,\rchi)
\end{align}

and 
\begin{align}
    \lambda_{\mathrm{sym}^jf}^2(n)l_3(n)\rchi(n) = O \left( \frac{x^{2+\epsilon}}{T}q^{\epsilon} + x^{\frac{3}{2}+\epsilon}q^{^{\frac{(j+1)^2}{4}+\epsilon}}T^{\frac{(j+1)^2}{4}-1+\epsilon} \right).
\end{align}
Thus \begin{align}
     \lambda_{\mathrm{sym}^jf}^2(n)r_6(n)\rchi(n) = \frac{\phi(4q)}{4q}b''_{j,f}(3)x^3 + O\left(\frac{x^{2+\epsilon}}{T}q^{\epsilon}+x^{\frac{3}{2}+\epsilon} T^{\frac{(j+1)^2}{4}-1+\epsilon} q^{^{\frac{(j+1)^2}{4}+\epsilon}}\right).
\end{align}

Now put $T = \frac{x^{\frac{2}{(j+1)^2}}}{q}$, then we have \begin{align}
     \lambda_{\mathrm{sym}^jf}^2(n)r_6(n)\rchi(n) = \frac{\phi(4q)}{4q}b''_{j,f}(3)x^3 + O\left(x^{3-\frac{2}{(j+1)^2}+\epsilon}q^{1+\epsilon}\right),
\end{align}
for $\rchi \rchi_4 \neq \rchi_0$.

We estimate $\lambda_{\mathrm{sym}^jf}^2(n)l_3(n)\rchi_0(n)$ and we have as proceeding as above,
\begin{align}
    \lambda_{\mathrm{sym}^jf}^2(n)r_6(n)\rchi_0(n) = \frac{\phi(q)}{q}b'_{j,f}(3)x^3 + O \left( x^{3-\frac{2}{(j+1)^2}+\epsilon}q^{1+\epsilon} \right),
\end{align}
where \begin{align} \label{eq:b'jf}
    b'_{j,f}(3) &= \frac{1}{3}L(3,\rchi_ 4\rchi_ 0)\prod_{n=1}^{j}L(3,\mathrm{sym}^{2n}f\otimes \rchi_ 4\rchi_ 0)L(1, \mathrm{sym}^{2n}f \otimes \rchi_ 0)H_j(3,\rchi_ 0).
\end{align}

For $\rchi\rchi_4 \neq \rchi_0$ and $\rchi \neq \rchi_0$, we have 
\begin{align}
     \lambda_{\mathrm{sym}^jf}^2(n)r_6(n)\rchi(n) = O\left(x^{3-\frac{2}{(j+1)^2}+\epsilon}q^{1+\epsilon} \right).
\end{align}

Similarly we have \begin{align}
    \lambda_{\mathrm{sym}^jf}^2(n)r_{10}(n)(n)\rchi_0(n) = \frac{64}{5}\frac{\phi(q)}{q}b'_{j,f}(5)x^5 +  O\left(x^{5-\frac{2}{(j+1)^2}+\epsilon}q^{1+\epsilon} \right),
\end{align}
and 
\begin{align}
     \lambda_{\mathrm{sym}^jf}^2(n)r_{10}(n)\rchi(n) = \frac{4}{5}\frac{\phi(4q)}{4q}b''_{j,f}(5)x^5 + O\left(x^{5-\frac{2}{(j+1)^2}+\epsilon}q^{1+\epsilon} \right),
\end{align}
where $\rchi \rchi_4 = \rchi_0$ and 

\begin{align}
     \lambda_{\mathrm{sym}^jf}^2(n)r_{10}(n)\rchi(n) = O\left(x^{5-\frac{2}{(j+1)^2}+\epsilon}q^{1+\epsilon} \right),
\end{align}
where $\rchi \rchi_4 \neq \rchi_0$.
\end{proof}.

\section{Proof of Theorem~\ref{thm:main}} \label{sec:main}

Let $\rchi$ be a Dirichlet character modulo $q$. By orthogonality, we have
\begin{align}
    &\sum_{\substack{n\leq x+1 \\ n \equiv 1~( \mathrm{mod }~q)}}\lambda_{\mathrm{sym}^jf}(n)l_1(n) = \frac{1}{\phi(q)} \sum_{\rchi(q)}\sum_{n\leq x+1} \lambda_{\mathrm{sym}^jf}(n)l_1(n)\rchi(n) \\
    &= \frac{1}{\phi(q)} \left \{ \sum_{n\leq x+1} \lambda_{\mathrm{sym}^jf}(n)l_1(n)\rchi_ 0(n) + \sum_{\substack{n \leq x+1 \\  \rchi_ 4\rchi = \rchi_ 0 }} \lambda_{\mathrm{sym}^jf}(n)l_1(n)\rchi(n)  \right. \\
    &+\left . \sum_{\substack{\rchi(q) \\ \rchi, \rchi_ 4\rchi \neq \rchi_ 0 \\ \rchi \text{ non primitive}}}\sum_{n\leq x+1} \lambda_{\mathrm{sym}^jf}(n)l_1(n)\rchi(n) +\sum_{\substack{\rchi(q) \\ \rchi_ 4, \rchi \neq \rchi_ 0 \\ \rchi \text{ non primitive}}}\sum_{n\leq x+1} \lambda_{\mathrm{sym}^jf}(n)l_1(n)\rchi(n)\right\} \\
    &= \frac{1}{\phi(q)}\left (\sum_1 + \sum_2 + \sum_3 + \sum_4\right).
\end{align}

Now by the Lemma~\ref{lem:r_2chi}, we have \begin{align}
    \sum_1+\sum_2 = O\left (x^{1-\frac{1}{j+1}+\epsilon}q^{\frac{1}{2}+\epsilon}\right)
\end{align}
and 
\begin{align}
    \sum_3+\sum_4 = O\left (x^{1-\frac{1}{j+1}+\epsilon}q^{1+\epsilon}\right).
\end{align}

Thus, we have \begin{align}
    \sum_{\substack{n\leq x+1 \\ n \equiv 1~( \mathrm{mod }~q)}}\lambda_{\mathrm{sym}^jf}(n)r_2(n) = O\left (\frac{x^{1-\frac{1}{j+1}+\epsilon}q^{1+\epsilon}}{\phi(q)}\right).
\end{align}

\section{Proof of Theorem~\ref{thm:main(4681012)}} \label{sec:main(4681012)}

 We first consider the sum of $4$ squares. By orthogonality, we have

\begin{align}
    \sum_{\substack{n\leq x+1 \\ n \equiv 1~( \mathrm{mod }~q)}}\lambda_{\mathrm{sym}^jf}(n)r_m(n) &= \frac{1}{\phi(q)} \sum_{\rchi(q)}\sum_{n\leq x+1} \lambda_{\mathrm{sym}^jf}(n)r_m(n)\rchi(n) \\
    &= \frac{1}{\phi(q)} \left \{ \sum_{n\leq x+1} \lambda_{\mathrm{sym}^jf}(n)r_m(n)\rchi_ 0(n) + \sum_{\substack{\rchi(q) \\ \rchi \neq \rchi_ 0 \\ \rchi \text{ primitive}}}\sum_{n\leq x+1} \lambda_{\mathrm{sym}^jf}(n)r_m(n)\rchi(n)  \right. \\
    &+\left . \sum_{\substack{\rchi(q) \\ \rchi \neq \rchi_ 0 \\ \rchi \text{ non primitive}}}\sum_{n\leq x+1} \lambda_{\mathrm{sym}^jf}(n)r_m(n)\rchi(n)\right\} \\
    &= \frac{1}{\phi(q)}\left (\sum_1 + \sum_2 + \sum_3 \right).
\end{align}

Now by Lemma~\ref{lem:r4681012chi}, we have \begin{align}
    \sum_1 = O\left( x^{\frac{m}{2}-\frac{2}{j+3}+\epsilon} \right)
\end{align}
and 
\begin{align}
    \sum_2+\sum_3 = O \left( x^{\frac{m}{2}-\frac{2}{j+3}+\epsilon}q^{\frac{j+1}{j+3}+\epsilon} \right).
\end{align}
Thus, we have \begin{align}
    \sum_{\substack{n\leq x+1 \\ n \equiv 1~( \mathrm{mod }~q)}}\lambda_{\mathrm{sym}^jf}(n)r_m(n) = O \left( \frac{x^{\frac{m}{2}-\frac{2}{j+3}+\epsilon}q^{\frac{j+1}{j+3}+\epsilon}}{\phi(q)} \right).
\end{align}

\section{Proof of Theorem~\ref{thm:main1}} \label{sec:main1}

We first consider the sum $\sum_{\substack{n\leq x \\ n \equiv 1 ~( \mathrm{mod }~q)}}\lambda_{\mathrm{sym}^jf}^2(n)l_1(n)$. By orthogonality, we have
\begin{align}
    &\sum_{\substack{n\leq x+1 \\ n \equiv 1~( \mathrm{mod }~q)}}\lambda_{\mathrm{sym}^jf}^2(n)l_1(n) = \frac{1}{\phi(q)} \sum_{\rchi(q)}\sum_{n\leq x+1} \lambda_{\mathrm{sym}^jf}^2(n)l_1(n)\rchi(n) \\
    &= \frac{1}{\phi(q)} \left \{ \sum_{n\leq x+1} \lambda_{\mathrm{sym}^jf}^2(n)l_1(n)\rchi_ 0(n) + \sum_{\substack{n \leq x+1 \\  \rchi_ 4\rchi = \rchi_ 0 }} \lambda_{\mathrm{sym}^jf}^2(n)l_1(n)\rchi(n)  \right. \\
    &+\left . \sum_{\substack{\rchi(q) \\ \rchi, \rchi_ 4\rchi \neq \rchi_ 0 \\ \rchi \text{primitive}}}\sum_{n\leq x+1} \lambda_{\mathrm{sym}^jf}^2(n)l_1(n)\rchi(n) +\sum_{\substack{\rchi(q) \\ \rchi_ 4, \rchi \neq \rchi_ 0 \\ \rchi \text{ non primitive}}}\sum_{n\leq x+1} \lambda_{\mathrm{sym}^jf}^2(n)l_1(n)\rchi(n)\right\} \\
    &= \frac{1}{\phi(q)}\left (\sum_1 + \sum_2 + \sum_3 + \sum_4\right).
\end{align}

Now by Lemma~\ref{lem:r_2chi2}, we have \begin{align}
    \sum_1 + \sum_2 = A \frac{\phi(q)}{q}x + O\left( x^{1-\frac{1}{(j+1)^2}+\epsilon}q^{\frac{1}{2}+\epsilon} \right)
\end{align}

and \begin{align}
    \sum_3+ \sum_4 = O\left( x^{1-\frac{1}{(j+1)^2}+\epsilon}q^{1+\epsilon} \right).
\end{align}

Thus, we have \begin{align}
    \sum_{\substack{n\leq x+1 \\ n \equiv 1~( \mathrm{mod }~q)}}\lambda_{\mathrm{sym}^jf}^2(n)r_2(n) = A \frac{1}{q}x + O\left( \frac{x^{1-\frac{1}{(j+1)^2}+\epsilon}q^{1+\epsilon}}{\phi(q)} \right).
\end{align}

\section{Proof of Theorem~\ref{thm:main1(468)}} \label{sec:main1(468)}

We first calculate $\sum_{\substack{n\leq x+1 \\ n \equiv 1~( \mathrm{mod }~q)}}\lambda_{\mathrm{sym}^jf}^2(n)r_m(n)$, for $m=4,8,12$. As above, \begin{align}
    \sum_{\substack{n\leq x+1 \\ n \equiv 1~( \mathrm{mod }~q)}}\lambda_{\mathrm{sym}^jf}^2(n)r_m(n) &= \frac{1}{\phi(q)} \sum_{\rchi(q)}\sum_{n\leq x+1} \lambda_{\mathrm{sym}^jf}^2(n)l_2(n)\rchi(n) \\
    &= \frac{1}{\phi(q)} \left \{ \sum_{n\leq x+1} \lambda_{\mathrm{sym}^jf}^2(n)r_m(n)\rchi_ 0(n) + \sum_{\substack{\rchi(q) \\ \rchi \neq \rchi_ 0 \\ \rchi \text{ primitive}}}\sum_{n\leq x+1} \lambda_{\mathrm{sym}^jf}^2(n)r_m(n)\rchi(n)  \right. \\
    &+\left . \sum_{\substack{\rchi(q) \\ \rchi \neq \rchi_ 0 \\ \rchi \text{ non primitive}}}\sum_{n\leq x+1} \lambda_{\mathrm{sym}^jf}^2(n)r_m(n)\rchi(n)\right\} \\
    &= \frac{1}{\phi(q)}\left (\sum_1 + \sum_2 + \sum_3 \right).
\end{align}

Now from Lemma~\ref{lem:r4812chi2}, we have \begin{align}
    \sum_1 = A \frac{\phi(q)}{q}x^{\frac{m}{2}}+O\left(x^{\frac{m}{2}-\frac{2}{(j+1)^2}+\epsilon} \right),
\end{align}
and \begin{align}
    \sum_2+\sum_3 = O\left(x^{\frac{m}{2}-\frac{2}{(j+1)^2}+\epsilon}q^{1+\epsilon}\right).
\end{align}

Again for $m=6,10$, we have \begin{align}
    &\sum_{\substack{n\leq x+1 \\ n \equiv 1~( \mathrm{mod }~q)}}\lambda_{\mathrm{sym}^jf}^2(n)r_m(n) = \frac{1}{\phi(q)} \sum_{\rchi(q)}\sum_{n\leq x+1} \lambda_{\mathrm{sym}^jf}^2(n)r_m(n)\rchi(n) \\
    &= \frac{1}{\phi(q)} \left \{ \sum_{n\leq x+1} \lambda_{\mathrm{sym}^jf}^2(n)r_m(n)\rchi_ 0(n) + \sum_{\substack{n \leq x+1 \\  \rchi_ 4\rchi = \rchi_ 0 }} \lambda_{\mathrm{sym}^jf}^2(n)r_m(n)\rchi(n)  \right. \\
    &+\left . \sum_{\substack{\rchi(q) \\ \rchi, \rchi_ 4\rchi \neq \rchi_ 0 \\ \rchi \text{ primitive}}}\sum_{n\leq x+1} \lambda_{\mathrm{sym}^jf}^2(n)r_m(n)\rchi(n) +\sum_{\substack{\rchi(q) \\ \rchi_ 4, \rchi \neq \rchi_ 0 \\ \rchi \text{ non primitive}}}\sum_{n\leq x+1} \lambda_{\mathrm{sym}^jf}^2(n)r_m(n)\rchi(n)\right\} \\
    &= \frac{1}{\phi(q)}\left (\sum_1 + \sum_2 + \sum_3 + \sum_4\right).
\end{align}

Now, from Lemma~\ref{lem:r610chi2}, we have \begin{align}
    \sum_1+\sum_2 = A\frac{\phi(q)}{q}x^{\frac{m}{2}}+O \left( x^{\frac{m}{2}-\frac{2}{(j+1)^2}+\epsilon}q^{1+\epsilon}\right)
\end{align}

and 
\begin{align}
    \sum_3+\sum_4 = O \left( x^{\frac{m}{2}-\frac{2}{(j+1)^2}+\epsilon}q^{1+\epsilon}\right).
\end{align}

Thus, in both cases, we have \begin{align}
    \sum_{\substack{n\leq x+1 \\ n \equiv 1~( \mathrm{mod }~q)}}\lambda_{\mathrm{sym}^jf}^2(n)r_m(n) = A \frac{1}{q}x^{\frac{m}{2}} + O\left( \frac{x^{\frac{m}{2}-\frac{2}{(j+1)^2}+\epsilon}q^{1+\epsilon}}{\phi(q)} \right).
\end{align}

\section{Proof of Theorem~\ref{thm:main2}} \label{sec:conv-sum}

By the definition of $k$-full kernel function $a(n)$, we can decompose $a(n)$ uniquely as $a(n) = k(n)q(n)$ with $(k(n),q(n)) = 1$, where $k(n)$ is $k$-full and $q(n)$ is $k$-free. Also note that $$ a(n) = a(k(n)) \ll n^\epsilon, \quad \lambda_{\mathrm{sym}^jf}(n) \ll n^\epsilon, \quad r_{m}(n) \ll n^{\frac{m}{2}+\epsilon},$$ for all $\epsilon > 0$. Here we will follow the procedure in \cite{Venkat}.

Let $1 < H \leq x^{\frac{2}{j+3}}$, then 
\begin{align}
    &\sum_{n \leq x} a(n) \lambda_{\mathrm{sym}^jf}(n+1) r_{m}(n+1) \\
    & = \sum_{\substack{n\leq x \\ k(n) \leq H}} a(n) \lambda_{\mathrm{sym}^jf}(n+1) r_{m}(n+1) + \sum_{\substack{n\leq x \\ k(n) > H}} a(n) \lambda_{\mathrm{sym}^jf}(n+1) r_{m}(n+1) \\
    & = \sum_{\substack{n\leq x \\ k(n) \leq H}} a(n) \lambda_{\mathrm{sym}^jf}(n+1) r_{m}(n+1) + O\left( \sum_{H < k(n) \leq x} a(k(n)) \sum_{\substack{q(n) \leq \frac{x}{k(n)} \\ (q(n),k(n)) = 1}}\lambda_{\mathrm{sym}^jf}(n+1) r_{m}(n+1) \right) \\
    & = I_1+I_2.
\end{align}

Now \begin{align}
    I_2 &= O\left( \sum_{H < k(n) \leq x}(k(n))^{\epsilon}\sum_{\substack{q(n) \leq \frac{x}{k(n)} \\ (q(n),k(n)) = 1}} (k(n)q(n))^{\frac{m}{2}-1+\epsilon}\right) \\
    & = O\left( \sum_{H < k(n) \leq x} (k(n))^{\frac{m}{2}-1+2\epsilon}\left( \frac{x}{k(n)}\right)^{\frac{m}{2}+\epsilon} \right) = O \left( x^{5+\epsilon}\sum_{H < k(n) \leq x}\frac{1}{(k(n))^{1-\epsilon}} \right) \\
    & = O(x^{\frac{m}{2}+\epsilon}H^{\frac{1}{k}-1+\epsilon}).
\end{align}

Here we have used the fact that $$ \# \{n \leq H \bigg| \text{ } n \text{ is $k$ - full \} } \ll H^{1/k}, \quad (\text{ by Lemma $2.12$ of \cite{Venkat} })$$ and $$ \sum_{\substack{n > H \\ n \text{ $k$-full}}}\frac{1}{n^{1-\epsilon}} \ll H^{\frac{1}{k}-1+\epsilon}, \quad (\text{ by Lemma $2.12$ of \cite{Venkat} }).$$

Now define $g(l) = \sum_{md^k=l}\mu(d)$, then $g(q(n) = 1$ as $q(n)$ is $k$-free. We have \begin{align}
    I_1 &= \sum_{\substack{n\leq x \\ k(n) \leq H}} a(n) \lambda_{\mathrm{sym}^jf}(n+1) r_{m}(n+1) \\
    & \ll \sum_{k(n) \leq H}a(k(n)) \sum_{\substack{q(n) \leq \frac{x}{k(n)}\\ (q(n),k(n)=1}}\lambda_{\mathrm{sym}^jf}(k(n)q(n)+1) r_{m}(k(n)q(n)+1) \\
    & \ll  \sum_{k(n) \leq H}a(k(n)) \sum_{\substack{q(n) \leq \frac{x}{k(n)}\\ (q(n),k(n)=1}}g(q(n))\lambda_{\mathrm{sym}^jf}(k(n)q(n)+1) r_{m}(k(n)q(n)+1) \\
    & \ll \sum_{k(n) \leq H}a(k(n)) \sum_{\substack{q(n) \leq \frac{x}{k(n)}\\ (q(n),k(n)=1}} \sum_{m(n)d^k(n) = q(n)} \mu(d(n)) \lambda_{\mathrm{sym}^jf}(k(n)q(n)+1) r_{m}(k(n)q(n)+1) \\
    & \ll \sum_{k(n) \leq H}a(k(n)) \sum_{\substack{d(n) \leq (\frac{x}{k(n)})^{1/k} \\ (d(n),k(n))=1}} \mu(d(n)) \sum_{\substack{m(n) \leq \frac{x}{k(n)d^k(n)} \\ (m(n),k(n) =1}} \lambda_{\mathrm{sym}^jf}(k(n)m(n)d^k(n)+1)r_{m}(k(n)m(n)d^k(n)+1) \\
    & = \sum_1^* + \sum_2^*,
    \end{align}

    where \begin{align}
        \sum_1^* & = \sum_{k(n) \leq H}a(k(n)) \sum_{\substack{d(n) \leq H^{1/k} \\ (d(n),k(n)) = 1}} \mu(d(n)) \\
        & \times \sum_{\substack{m(n) \leq \frac{x}{k(n)d^k(n)} \\ (m(n),k(n) =1}}\lambda_{\mathrm{sym}^jf}(k(n)m(n)d^k(n)+1)r_{m}(k(n)m(n)d^k(n)+1)
    \end{align}

    and \begin{align}
        \sum_2^* & = \sum_{k(n) \leq H}a(k(n)) \sum_{\substack{H^{1/k} < d(n) \leq (\frac{x}{k(n)})^{1/k} \\ (d(n),k(n)) = 1}} \mu(d(n)) \\
        & \times \sum_{\substack{m(n) \leq \frac{x}{k(n)d^k(n)} \\ (m(n),k(n) =1}}\lambda_{\mathrm{sym}^jf}(k(n)m(n)d^k(n)+1)r_{m}(k(n)m(n)d^k(n)+1).
    \end{align}
    Now for $\sum_2^*$, we have
    \begin{align}
        \sum_2^* &\ll \sum_{k(n) \leq H}(k(n))^\epsilon \sum_{d(n) \geq H^{1/k}} \sum_{\substack{m(n) \leq \frac{x}{k(n)d^k(n)} \\ (m(n),k(n) =1}}(k(n)m(n)d^k(n))^{\frac{m}{2}-1+\epsilon} \\
        & \ll \sum_{k(n) \leq H}(k(n))^{\frac{m}{2}-1+2\epsilon}\sum_{d(n) \geq H^{1/k}}(d^k(n))^{\frac{m}{2}-1+\epsilon}\left(\frac{x}{k(n)d^k(n)}\right)^{\frac{m}{2}+\epsilon} \\
        & \ll x^{\frac{m}{2}+\epsilon}\sum_{k(n) \leq H} \frac{1}{(k(n))^{1-\epsilon}} \sum_{d^k(n) \geq H}\frac{1}{d^k(n)} \\
        & \ll x^{\frac{m}{2}+\epsilon}[H^{\epsilon}+\zeta(1-\epsilon)+O(1/H^{1-\epsilon})]H^{1/k-1} \\
        & \ll x^{\frac{m}{2}+\epsilon}H^{1/k-1+\epsilon}.
    \end{align}

    Now note that $$\sum_{\delta(n)|(m(n),k(n))} \mu(\delta(n)) = \begin{cases}
        1 \text{ if } (m(n),k(n))= 1, \\ 0 \text{ otherwise}.
    \end{cases}$$

    So we have for $\sum_1^*$, \begin{align} \label{eq:sum_conv}
        \sum_1^* &= \sum_{k(n) \leq H}a(k(n)) \sum_{\substack{d(n) \leq H^{1/k} \\ (d(n),k(n)) = 1}} \mu(d(n)) \sum_{\delta(n)|k(n)} \mu(\delta(n)) \\
        & \times \sum_{m_1(n)\delta(n)k(n)d^k(n) \leq x}\lambda_{\mathrm{sym}^jf}(k(n)m_1(n)\delta(n)d^k(n)+1)r_{m}(k(n)m_1(n)\delta(n)d^k(n)+1).
    \end{align}

    Here $m_1(n)$ varies, so we can write, \begin{align} \label{eq:m1(n)}
        &\sum_{\substack{m_1(n)\delta(n)k(n)d^k(n) \leq x \\}}\lambda_{\mathrm{sym}^jf}(~k(n)m_1(n)\delta(n)d^k(n)+1~)r_{m}(~k(n)m_1(n)\delta(n)d^k(n)+1~) \\ 
        & := \sum_{\substack{R \leq x+1 \\ R \equiv 1 ~(\mathrm{ mod }~ k(n)\delta(n)d^k(n))}}\lambda_{\mathrm{sym}^jf}(R)r_{m}(R) \\
        & =  O\left( \frac{x^{\frac{m}{2}+\epsilon-\frac{2}{j+3}}(k(n)\delta(n)d^k(n))^{\frac{j+1}{j+3}+\epsilon}}{\phi(k(n)\delta(n)d^k(n))} \right) \quad \text{(by Theorem }~\ref{thm:main}) \\
        & =  O\left( \frac{x^{\frac{m}{2}+\epsilon-\frac{2}{j+3}}(k(n)\delta(n)d^k(n))^{1+\epsilon}}{\phi(k(n)\delta(n)d^k(n))} \right) .
    \end{align}

    So \begin{align}
        \sum_1^* &\ll  \sum_{k(n) \leq H}a(k(n)) \sum_{\substack{d(n) \leq H^{1/k} \\ (d(n),k(n)) = 1}} \mu(d(n)) \sum_{\delta(n)|k(n)} \mu(\delta(n)) \frac{x^{\frac{m}{2}+\epsilon-\frac{2}{j+3}}(k(n)\delta(n)d^k(n))^{1+\epsilon}}{\phi(k(n)\delta(n)d^k(n))} \\
        & \ll \sum_{k(n) \leq H}(k(n))^{\epsilon} \sum_{\substack{d(n) \leq H^{1/k} \\ (d(n),k(n)) = 1}}1 \sum_{\delta(n)|k(n)}x^{\frac{m}{2}+\epsilon-\frac{2}{j+3}}(k(n)\delta(n)d^k(n))^{\epsilon} \frac{k(n)\delta(n)}{\phi(k(n)\delta(n))} \frac{d^k(n)}{\phi(d^k(n))} \\
        & \ll x^{\frac{m}{2}+4\epsilon-\frac{2}{j+3}}  \sum_{k(n) \leq H} (k(n))^\epsilon  \sum_{\substack{d(n) \leq H^{1/k}}}1 \sum_{\delta(n)|k(n)} \log \log (k(n)\delta(n)) \log \log (d^k(n)) \\
        & \ll x^{\frac{m}{2}+4\epsilon-\frac{2}{j+3}} \log \log H^2 \log \log H \sum_{k(n) \leq H} (k(n))^\epsilon \sum_{\substack{d(n) \leq H^{1/k}}}1 \sum_{\delta(n)|k(n)}1 \\
        & \ll x^{\frac{m}{2}+4\epsilon-\frac{2}{j+3}} \log \log H^2 \log \log H \sum_{k(n) \leq H} (k(n))^{2\epsilon}\sum_{\substack{d(n) \leq H^{1/k}}}1 \\
        & \ll x^{\frac{m}{2}+4\epsilon-\frac{2}{j+3}} \log \log H^2 \log \log H \sum_{k(n) \leq H} (k(n))^{2\epsilon}H^{1/k} \\
        & \ll x^{\frac{m}{2}+4\epsilon-\frac{2}{j+3}}H^{2\epsilon} H^{1/k}\log \log H^2 \log \log H\sum_{k(n) \leq H}1  \\
        &\ll x^{\frac{m}{2}+\epsilon-\frac{2}{j+3}}H^{2/k+\epsilon}.
    \end{align}

    Thus combining everything, we have \begin{align}
        \sum_{n \leq x} a(n) \lambda_{\mathrm{sym}^jf}(n+1)r_{m}(n+1) = O\left( x^{\frac{m}{2}+\epsilon}H^{1/k-1+\epsilon} + x^{\frac{m}{2}-\frac{2}{j+3}+\epsilon}H^{2/k+\epsilon} \right).
    \end{align}

    Note that $q = \delta(n)k(n)d^k(n)$ and $\delta(n), k(n), d^k(n) \leq H$, so $q \leq H^3$. Also $x$ is a very large number, so we can consider $q \ll x^{\frac{2}{j+1}-\epsilon}$. Thus we choose a optimal value of $H$, which is $H = x^{\frac{2}{3(j+1)}-\epsilon}$, and we obtain \begin{align}
        \sum_{n \leq x} a(n) \lambda_{\mathrm{sym}^jf}(n+1)r_{m}(n+1) &= O\left( x^{\frac{m}{2}-\frac{2k-2}{3k(j+1)}+\epsilon} + x^{\frac{m}{2}+\frac{4}{3k(j+1)}-\frac{2}{j+3} + \epsilon} \right) = O \left( x^{\frac{m}{2}-\frac{2k-2}{3k(j+1)}+\epsilon} \right), 
    \end{align}
    which is true for $j \geq 1$. Thus we have our result for $m=4,6,8,10,12$.

    Now for $m=2$, it will be same upto \eqref{eq:sum_conv} and we have from \eqref{eq:m1(n)}, \begin{align}
        &\sum_{\substack{m_1(n)\delta(n)k(n)d^k(n) \leq x \\}}\lambda_{\mathrm{sym}^jf}(k(n)m_1(n)\delta(n)d^k(n)+1)r_{2}(k(n)m_1(n)\delta(n)d^k(n)+1) \\ 
        & := \sum_{\substack{R \leq x+1 \\ R \equiv 1 ~(\mathrm{ mod ~}k(n)\delta(n)d^k(n))}}\lambda_{\mathrm{sym}^jf}(R)r_{2}(R) \\
        & =  O\left( \frac{x^{1+\epsilon-\frac{1}{j+1}}(k(n)\delta(n)d^k(n))^{1+\epsilon}}{\phi(k(n)\delta(n)d^k(n))} \right) \quad \text{( by theorem }~\ref{thm:main}).
    \end{align}
    So as in the previous calculation, we have \begin{align}
        \sum_1^* &\ll \sum_{k(n) \leq H}a(k(n)) \sum_{\substack{d(n) \leq H^{1/k} \\ (d(n),k(n)) = 1}} \mu(d(n)) \sum_{\delta(n)|k(n)} \mu(\delta(n)) \frac{x^{1+\epsilon-\frac{1}{j+1}}(k(n)\delta(n)d^k(n))^{1+\epsilon}}{\phi(k(n)\delta(n)d^k(n))} \\
        &\ll x^{1+\epsilon-\frac{1}{j+1}}H^{2/k+\epsilon}.
    \end{align}

    Thus, combining everything, we have \begin{align}
        \sum_{n \leq x} a(n) \lambda_{\mathrm{sym}^jf}(n+1)r_{2}(n+1) = O\left( x^{1+\epsilon}H^{1/k-1+\epsilon} + x^{1-\frac{1}{j+1}+\epsilon}H^{2/k+\epsilon} \right).
    \end{align}
As above, we note that $q \ll x^{\frac{1}{j+1}-\epsilon}$ and we take $H = x^{\frac{1}{3(j+1)}-\epsilon}$. Thus, we have \begin{align}
    \sum_{n \leq x} a(n) \lambda_{\mathrm{sym}^jf}(n+1)r_{2}(n+1) = O\left( x^{1-\frac{k-1}{3k(j+1)}+\epsilon} \right).
\end{align}

\section{Proof of Theorem~\ref{thm:main3}} \label{sec:conv-sum-square}

Note that $a^2(n) = a^2(k(n)$ and $a^2(n) \ll n^\epsilon$ for all $\epsilon > 0$. Now to prove this theorem, we will follow the same steps up to \eqref{eq:sum_conv} in the Theorem~\ref{thm:main2}. Then we have \begin{align} \label{eq:m11(n)}
    &\sum_{\substack{m_1(n)\delta(n)k(n)d^k(n) \leq x \\}}\lambda_{\mathrm{sym}^jf}^2(k(n)m_1(n)\delta(n)d^k(n)+1)r_{m}(k(n)m_1(n)\delta(n)d^k(n)+1) \\ 
        & := \sum_{\substack{R \leq x+1 \\ R \equiv 1 ~( \mathrm{mod~}k(n)\delta(n)d^k(n))}}\lambda_{\mathrm{sym}^jf}^2(R)r_{m}(R) \\
        & =  \frac{c_{j,f}(\frac{m}{2})}{k(n)\delta(n)d^k(n)}x^\frac{m}{2}+O\left( \frac{x^{\frac{m}{2}+\epsilon-\frac{2}{(j+1)^2}}(k(n)\delta(n)d^k(n))^{1+\epsilon}}{\phi(k(n)\delta(n)d^k(n))} \right), \quad (\text{ by Theorem~\ref{thm:main1}}).
\end{align}

So we have \begin{align}
    \sum_{1}^* = \sum_1^{'} + \sum_1^{"},
\end{align}

where \begin{align}
    \sum_1^{'} = \sum_{k(n) \leq H}a^2(k(n)) \sum_{\substack{d(n) \leq H^{1/k} \\ (d(n),k(n)) = 1}} \mu(d(n)) \sum_{\delta(n)|k(n)} \mu(\delta(n)) \frac{c_{j,f}(\frac{m}{2})x^{\frac{m}{2}}}{\delta(n)k(n)d^k(n)},
\end{align}

and proceeding as in the previous theorem, we have \begin{align}
    \sum_{1}^{''} \ll x^{\frac{m}{2}+\epsilon-\frac{2}{(j+1)^2}}H^{2/k+\epsilon}.
\end{align}

Now \begin{align}
    \sum_1^{'} &= c_{j,f}(\frac{m}{2})x^\frac{m}{2}\sum_{k(n) \leq H} \frac{a^2(k(n))}{k(n)} \sum_{\delta(n)|k(n)} \frac{\mu(\delta(n))}{\delta(n)}\sum_{\substack{d(n) \leq H^{1/k} \\ (d(n),k(n)) = 1}} \frac{\mu(d(n))}{d^k(n)} \\
    &= c_{j,f}(\frac{m}{2})x^\frac{m}{2}\sum_{k(n)=1}^{\infty} \frac{a^2(k(n))}{k(n)} \sum_{\delta(n)|k(n)} \frac{\mu(\delta(n))}{\delta(n)}\sum_{\substack{d(n)=1}}^{\infty} \frac{\mu(d(n))}{d^k(n)} \\
    &+ O \left( x^\frac{m}{2}\sum_{k(n)=1}^{\infty} \frac{a^2(k(n))}{k(n)} \sum_{\delta(n)|k(n)} \frac{\mu(\delta(n))}{\delta(n)}\sum_{\substack{d(n) > H^{1/k} \\ (d(n),k(n)) = 1}} \frac{\mu(d(n))}{d^k(n)}\right) \\
    &+ O \left( x^\frac{m}{2}\sum_{k(n) > H} \frac{a^2(k(n))}{k(n)} \sum_{\delta(n)|k(n)} \frac{\mu(\delta(n))}{\delta(n)}\sum_{\substack{d(n) \leq H^{1/k} \\ (d(n),k(n)) = 1}} \frac{\mu(d(n))}{d^k(n)}\right) \\
    & = c_{j,f}(\frac{m}{2})x^\frac{m}{2}\sum_{k(n)=1}^{\infty} \frac{a^2(k(n))}{k(n)} \sum_{\delta(n)|k(n)} \frac{\mu(\delta(n))}{\delta(n)}\sum_{\substack{d(n)=1}}^{\infty} \frac{\mu(d(n))}{d^k(n)} +J_1+J_2.
\end{align}

Now \begin{align}
    J_1 &\ll x^\frac{m}{2}\sum_{k(n)=1}^{\infty}\frac{(k(n))^{\epsilon}}{k(n)} \sum_{\delta(n)|k(n)} \frac{1}{\delta(n)}\sum_{\substack{d(n) > H^{1/k}}} \frac{1}{d^k(n)} \\
    & \ll x^\frac{m}{2}\sum_{k(n)=1}^{\infty}\frac{1}{k(n)^{1-\epsilon}} \sum_{\delta(n)|k(n)} 1\sum_{\substack{d^k(n) > H}} \frac{1}{d^k(n)}  \ll x^\frac{m}{2}\sum_{k(n)=1}^{\infty}\frac{1}{k(n)^{1-2\epsilon}}H^{1/k-1} \\
    & \ll x^\frac{m}{2} \sum_{n=1}^{\infty}\frac{1}{n^{k-2k\epsilon}}H^{1/k-1}   \ll x^\frac{m}{2}H^{1/k-1} \quad ( \text{ as } k(n) \text{ is $k$-full }).
\end{align}

Again \begin{align}
    & J_2 \ll x^\frac{m}{2} \sum_{k(n) > H}\frac{1}{(k(n))^{1-\epsilon}}\sum_{\delta(n)|k(n)}1\sum_{d(n) \leq H^{1/k}}\frac{1}{d^k(n)} \\
    &\ll x^\frac{m}{2}H^{1/k-1+2\epsilon} \left(\sum_{n=1}^{\infty}\frac{1}{n^k}+\sum_{d^k(n) > H} \frac{1}{d^k(n)}\right)  \ll x^\frac{m}{2}H^{1/k-1+\epsilon} + x^5H^{2/k-2+2\epsilon}  \ll x^\frac{m}{2}H^{1/k-1+\epsilon}.
\end{align}

Thus we have $J_1+J_2 = O(x^{\frac{m}{2}+\epsilon}H^{1/k-1+\epsilon})$ and the first term is finite. So \begin{align}
    \sum_1^{'} = C_{j,f}x^\frac{m}{2}+O(x^{\frac{m}{2}+\epsilon}H^{1/k-1+\epsilon}). 
\end{align}

Thus, combining everything, we have \begin{align}
    \sum_{n \leq x} a(n) \lambda_{\mathrm{sym}^jf}^2(n+1)r_{m}(n+1) &= C_{j,f}x^\frac{m}{2}+O(x^{\frac{m}{2}+\epsilon}H^{1/k-1+\epsilon} + x^{\frac{m}{2}+\epsilon-\frac{2}{(j+1)^2}}H^{2/k+\epsilon}).
\end{align}

 we now choose $H = x^{\frac{2}{3(j+1)^2}-\epsilon}$ and we obtain \begin{align}
    \sum_{n \leq x} a(n) \lambda_{\mathrm{sym}^jf}^2(n+1)r_{m}(n+1) &= C_{j,f}x^\frac{m}{2}+O \left( x^{\frac{m}{2}-\frac{2k-2}{3(j+1)^2k}+\epsilon} \right).
\end{align}

So the result happens for $m=4,6,8$.

Now for $m=2$, we have from \eqref{eq:m11(n)}, \begin{align}
    &\sum_{\substack{m_1(n)\delta(n)k(n)d^k(n) \leq x \\}}\lambda_{\mathrm{sym}^jf}^2(k(n)m_1(n)\delta(n)d^k(n)+1)r_{2}(k(n)m_1(n)\delta(n)d^k(n)+1) \\ 
        & := \sum_{\substack{R \leq x+1 \\ R \equiv 1 ~(\mathrm{mod}~k(n)\delta(n)d^k(n))}}\lambda_{\mathrm{sym}^jf}^2(R)r_{2}(R) \\
        & =  \frac{c_{j,f}(1)}{k(n)\delta(n)d^k(n)}x+O\left( \frac{x^{1+\epsilon-\frac{1}{(j+1)^2}}(k(n)\delta(n)d^k(n))^{1+\epsilon}}{\phi(k(n)\delta(n)d^k(n))} \right), \quad (\text{ by Theorem~\ref{thm:main1}}),
\end{align}

and proceeding as in the previous case, we have \begin{align}
    \sum_1' &= D_{f,j}x+O\left( x^{1+\epsilon}H^{1/k-1+\epsilon} \right)
\end{align}

and \begin{align}
    \sum_{1}^{''} \ll x^{1-\frac{1}{(j+1)^2}+\epsilon}H^{2/k+\epsilon}
\end{align}

Thus, combining everything, we have \begin{align}
    \sum_{n \leq x} a(n) \lambda_{\mathrm{sym}^jf}^2(n+1)r_{2}(n+1) &= C_{j,f}x+O(x^{1+\epsilon}H^{1/k-1+\epsilon} + x^{1+\epsilon-\frac{1}{(j+1)^2}}H^{2/k+\epsilon}).
\end{align}

 we now choose $H = x^{\frac{2}{3(j+1)^2}-\epsilon}$ and we obtain \begin{align}
    \sum_{n \leq x} a(n) \lambda_{\mathrm{sym}^jf}^2(n+1)r_{2}(n+1) &= C_{j,f}x+O \left( x^{1-\frac{k-1}{3(j+1)^2k}+\epsilon} \right).
\end{align}

\section{Proof of Theorem~\ref{thm:mainsign}} \label{sec:sign-change-2}

Let $S(x) = \sum_{n \leq x} a(n)\lambda_{\mathrm{sym}^jf}(n+1)r_2(n+1)$ and  $h=x^{\delta_j}$ with $A(j)\coloneqq 1-\frac{k-1}{3k(j+1)^2}  < \delta_j < 1$. Now suppose that $\{ a(n)\lambda_{\mathrm{sym}^jf}(n+1) | n = a_1^2 + a_2^2 - 1, a_i \in \mathbb{Z} \}$ does not change any sign in the interval $n \in (x,x+h]$ and without loss of generality suppose that the sequence stays positive in the given interval.

Using the second result of Theorem~\ref{thm:main2}, we have
\begin{align} \label{eq:uppbound}
    &\sum_{x < n \leq x+h}a^2(n)\lambda_{\mathrm{sym}^jf}^2(n+1)r_2(n+1)\\
    &= \sum_{x < n \leq x+h}a(n)\lambda_{\mathrm{sym}^jf}(n+1)a(n)\lambda_{\mathrm{sym}^jf}(n+1)r_2(n+1) \\
    & \ll (x+h)^{\epsilon} \sum_{x < n \leq x+h}a(n)\lambda_{\mathrm{sym}^jf}(n+1)r_2(n+1) \quad (\text{as } a(n)\lambda_{\mathrm{sym}^jf}(n+1) \ll n^\epsilon \text{ for any } \epsilon > 0) \\
    &\ll x^\epsilon (|S(x+h)|+|S(x)|) \\
    & \ll x^{1-\frac{k-1}{3k(j+1)}+\epsilon},
\end{align}
for any $\epsilon>0$, as small as possible.

Now using the second result of Theorem~\ref{thm:main3}, we have 
\begin{align}
    \sum_{x < n \leq x+h}a^2(n)\lambda_{\mathrm{sym}^jf}^2(n+1)r_2(n+1) & = D_{j,f,2}h + O\left( x^{1-\frac{k-1}{3k(j+1)^2}+\epsilon}\right) \\
    &= D_{j,f,2}x^{\delta_j} + O\left( x^{1-\frac{k-1}{3k(j+1)^2}+\epsilon}\right).
\end{align}

Lemma~\ref{dominating main term} ensures that 
\begin{align}\label{lower bound}
    \sum_{x < n \leq x+h}a^2(n)\lambda_{\mathrm{sym}^jf}^2(n+1)r_2(n+1)  \gg  x^{\delta_j}.
\end{align}
Combining \eqref{eq:uppbound} and \eqref{lower bound}, we obtain
$$x^{\delta_j }\ll x^{1-\frac{k-1}{3k(j+1)}+\epsilon}$$
as $x\to \infty$ for any $\epsilon >0.$ That is, $x^{\delta_j -A(j)-\epsilon}\ll 1$ as $x\to \infty$ for any $\epsilon >0.$ In particular, choosing $\epsilon=\frac{1}{2}(\delta_j -A(j))>0$, we obtain that $x^{\epsilon} \ll 1$ as $x \to \infty$, which is a contradiction.

This implies that there exists at least one sign change in the interval $(x, x+x^{\delta_j}]$, where $x$ is sufficiently large. Similarly, we can prove that there exists at least one sign change in $(x+x^{\delta_j}, x+2x^ {\delta_j}]$ and so on.

Note that $2x = x+ x^{1-\delta_j}x^{\delta_j}$, and we have that there exists at least $x^{1-\delta_j}$ number of sign changes in the interval $(x,2x]$.

\section{Proof of Theorem~\ref{thm:mainsign1}} \label{sec:sign-change-m}

Let $S_m(x) = \sum_{n \leq x} a(n)\lambda_{\mathrm{sym}^jf}(n+1)r_m(n+1)$ and  $h=x^{\delta_j}$ with $1-\frac{2k-2}{3k(j+1)^2}  < \delta_j < 1-\frac{k-1}{3k(j+1)^2}$. Now suppose that $\{ a(n)\lambda_{\mathrm{sym}^jf}(n+1) | n = \sum_{i=1}^{m}a_i^2-1, a_i \in \mathbb{Z} \}$ does not change any sign in the interval $n \in (x,x+h]$ and without loss of generality suppose that the sequence stays positive in the given interval.

Using the first result of Theorem~\ref{thm:main2}, we have
\begin{align} \label{eq:uppbound1}
    &\sum_{x < n \leq x+h}a^2(n)\lambda_{\mathrm{sym}^jf}^2(n+1)r_m(n+1) \\ &= \sum_{x < n \leq x+h}a(n)\lambda_{\mathrm{sym}^jf}(n+1)a(n)\lambda_{\mathrm{sym}^jf}(n+1)r_m(n+1) \\
    & \ll (x+h)^{\epsilon} \sum_{x < n \leq x+h}a(n)\lambda_{\mathrm{sym}^jf}(n+1)r_m(n+1) \quad (\text{ as } a(n)\lambda_{\mathrm{sym}^jf}(n+1) \ll n^\epsilon, \text{ for all } \epsilon > 0) \\
    &\ll x^\epsilon (|S_m(x+h)|+|S_m(x)|) \\
    & \ll x^{\frac{m}{2}-\frac{2k-2}{3k(j+1)}+\epsilon}.
\end{align}

Now using the first result of Theorem~\ref{thm:main3}, we have 
\begin{align}
    &\sum_{x < n \leq x+h}a(n)^2\lambda_{\mathrm{sym}^jf}^2(n+1)r_m(n+1) \\ & = D_{j,f,m}(x+h)^\frac{m}{2}-D_{j,f,m}x^{\frac{m}{2}} + O\left( x^{\frac{m}{2}-\frac{2}{(j+1)^2}+\epsilon}\right) \\
    &= D_{j,f,m}(c_1x^{\frac{m}{2}-1+\delta_j}+c_2x^{\frac{m}{2}-2+2\delta_j} + \cdots +c_{\frac{m}{2}}x^{\frac{m}{2}\delta_j})+ O\left( x^{\frac{m}{2}-\frac{2}{(j+1)^2}+\epsilon}\right)\\
    &= D_{j,f,m}'x^{\frac{m}{2}-1+\delta_j} + O\left( x^{\frac{m}{2}-\frac{2}{(j+1)^2}+\epsilon}\right) \quad  \left(\text{ as } \delta_j < 1-\frac{k-1}{(j+1)^2} \right ),
\end{align}
where $D_{j,f,m}'$ is a constant, depending on $j,f$ and $m$.

Thus, by the Lemma~\ref{dominating main term},  we have \begin{align} \label{eq:lowbound1}
    \sum_{x < n \leq x+h}\lambda_{\mathrm{sym}^jf}^2(n)r_m(n) \gg x^{\frac{m}{2}-1+\delta_j}.
\end{align}

Combining \eqref{eq:uppbound1} and \eqref{eq:lowbound1}, we have \begin{align}
    x^{\frac{m}{2}-1+\delta_j} \ll x^{\frac{m}{2}-\frac{2k-2}{3k(j+1)}+\epsilon}.
\end{align}
This gives \begin{align}
    x^{\frac{2k-2}{3k(j+1)}-\frac{2k-2}{3k(j+1)^2}+\epsilon} \leq x^{\frac{m}{2}-1+\delta_j-\frac{m}{2}+\frac{2}{j+1}+\epsilon} \ll 1,
\end{align}
for any $\epsilon > 0$, as small as possible. This leads to a contradiction as $x \mapsto +\infty$.

This implies that there exists at least one sign change in the interval $(x, x+x^{\delta_j}]$, where $x$ is sufficiently large. Similarly, we can prove that there exists at least one sign change in the interval $(x+x^{\delta_j}, x+2x^{\delta_j}]$ and so on. Thus, we have that there exists at least $x^{1-{\delta_j}}$ number of sign changes in the interval $(x,2x]$.

\bibliographystyle{plain}
\bibliography{references}

\end{document}